\documentclass[oneside,reqno]{amsart}
\usepackage{amssymb}
\usepackage{amsmath}
\usepackage{amsthm}
\usepackage{amsbsy}
\usepackage{enumerate}
\usepackage{bm}
\usepackage{hyperref}
\date{\today}
\usepackage{cite}
\usepackage{array}
\usepackage{xcolor}
\usepackage{tikz}
\usepackage{graphics}
\usepackage{subcaption}
\usepackage{mathtools}
\usepackage{relsize}

\hypersetup{
	colorlinks=true, 
	linkcolor=blue,
	filecolor=dviolet,
	urlcolor=blue,}
\makeatletter
\newcommand{\distas}[1]{\mathbin{\overset{#1}{\kern\z@\sim}}}

\newcommand{\distras}[1]{
  \mathbin{\overset{#1}{\kern\z@\resizebox{\wd\mybox}{\ht\mysim}{$\sim$}}}
}

    \newtheorem{theorem}{Theorem}
    \newtheorem{lemma}[theorem]{Lemma}
    \newtheorem{proposition}[theorem]{Proposition}

\theoremstyle{definition} % For roman text in the body
    \newtheorem{definition}[theorem]{Definition}

    \newtheorem{remark}[theorem]{Remark}
    \newtheorem{example}[theorem]{Example}
    \newtheorem{exercise}[theorem]{Exercise}

\newcommand{\E}{\mbox{E}}

\def\Z{\mathbb{Z}}

\def\N{\mathbb{N}}
\def\PP{{\mathcal P}}

\def\R{\mathbb{R}}

\def\Z{\mathbb{Z}}

\def\u{{\bf u}}

\def\tends{\rightarrow}

\def\<{\langle}
\def\>{\rangle}

\def\bar{\overline}

\newcommand\Tr{{\mbox{Tr}}}

\newcommand\mnote[1]{} %off
\newcommand\be{\begin{equation*}}

\newcommand\ee{\end{equation*}}

\newcommand\ben{\begin{equation}}
\newcommand\een{\end{equation}}
\newcommand\bes{\begin{eqnarray*}}
\newcommand\ees{\end{eqnarray*}}

\newcommand\bex{\begin{exercise}}
\newcommand\eex{\end{exercise}}
\newcommand\beg{\begin{example}}
\newcommand\eeg{\end{example}}
\newcommand\benu{\begin{enumerate}}
\newcommand\eenu{\end{enumerate}}
\newcommand\beit{\begin{itemize}}
\newcommand\eeit{\end{itemize}}
\newcommand\berk{\begin{remark}}
\newcommand\eerk{\end{remark}}
\newcommand\bdefn{\begin{defintion}}
\newcommand\edefn{\end{definition}}
\newcommand\bthm{\begin{theorem}}
\newcommand\ethm{\end{theorem}}
\newcommand\bprf{\begin{proof}}
\newcommand\eprf{\end{proof}}
\newcommand\blem{\begin{lemma}}
\newcommand\elem{\end{lemma}}

\newcommand{\var}{\mbox{\rm Var}}

\newcommand{\Cov}{\mbox{\rm Cov}}
\newcommand{\sm}{{\raise0.3ex\hbox{$\scriptstyle \setminus$}}}

\def\tends{\rightarrow}

\def\CHI{\mathchoice%
{\raise2pt\hbox{$\chi$}}%
{\raise2pt\hbox{$\chi$}}%
{\raise1.3pt\hbox{$\scriptstyle\chi$}}%
{\raise0.8pt\hbox{$\scriptscriptstyle\chi$}}}
\def\smalloplus{\raise1pt\hbox{$\,\scriptstyle \oplus\;$}}

\usepackage{tikz}

\title [Linear eigenvalue statistics of $XX'$] 
{Linear eigenvalue statistics of $XX^\prime$ matrices} % for Toeplitz matrices}

\author{Kiran Kumar A.S$^{\ast}$, Shambhu Nath Maurya$^{\dagger}$, Koushik Saha$^{\ddagger}$\\\\
	$^{\ast}$$^{\ddagger}$Department of Mathematics, 
	Indian Institute of Technology Bombay\\ Mumbai 400076, India \\\\
	$^{\dagger}$Statistics and Mathematics Unit,
	Indian Statistical Institute\\ Kolkata 700108, India}

\date{\today}
\thanks{ $^{\ast}$kiran [at] math.iitb.ac.in, $^{\dagger}$shambhumath4 [at] gmail.com, $^{\ddagger}$koushik.saha [at] iitb.ac.in}  
\begin{document}

\begin{abstract}

%This article focuses on linear eigenvalue statistics of some sample covariance matrices $XX^t$, where $X$ is the Toeplitz matrix $(T_{n\times p})$ or Hankel matrix $(H_{n\times p})$ with independent entries and some moment conditions. For $y \in (0, \infty)$, we show that the linear eigenvalue statistics of these matrices for polynomial test functions converge in distribution to Gaussian random variables when $p/n \tends y$. We find a closed form of variance formula for the limiting distributions. To establish the results, we use method of moments.
This article focuses on the fluctuations of linear eigenvalue statistics of $T_{n\times p}T'_{n\times p}$, where $T_{n\times p}$ is an $n\times p$ Toeplitz matrix with real, complex or time-dependent entries. We show that as $n \rightarrow
\infty$ and $p/n \rightarrow \lambda \in (0, \infty)$, the linear eigenvalue statistics of these matrices for polynomial test functions converge in distribution to Gaussian random variables. We also discuss the linear eigenvalue statistics of  $H_{n\times p}H'_{n\times p}$, when $H_{n\times p}$ is an $n\times p$ Hankel matrix. As a result of our studies, we also derive in-probability limit and a central limit theorem type result for Schettan norm of rectangular Toeplitz matrices. To establish the results, we use method of moments. 
\end{abstract}

\maketitle


\section{\large \bf \textbf{Introduction and main results}}

%The study of random Hankel matrices finds its genesis in \cite{bai_zd_99}. Here, Bai proposed the study of limiting spectral distributions(LSD) of Hankel and the closely related Toeplitz matrix. The existence of LSD for Hankel matrices was shown independently by Hammond and Miller  \cite{hammond_miller_05} and Bryc et al. \cite{bryc_lsd_06} for i.i.d. input sequences with finite variance.  

The study of linear eigenvalue statistics is a popular area of research in random matrix theory. For an $n \times n$ matrix $A_n$, the linear eigenvalue statistics of $A_n$ is defined as 
\begin{equation} \label{eqn:LES}
\mathcal{A}_n(\phi)= \sum_{i=1}^{n} \phi(\lambda_i),
\end{equation} 
where $\lambda_1 , \lambda_2 , \ldots , \lambda_n$ are the eigenvalues of $A_n$ and $\phi$ is a `nice' test function. The studies on linear eigenvalue statistics started with the study of central limit theorems for linear eigenvalue statistics of sample covariance matrix \cite{arharov} and \cite{jonsson}. In 2004,  Bai and Silverstein \cite{bai2004clt} provided a CLT for linear eigenvalue statistics of large dimensional sample covariance matrices with general entries, using a technique based on the Stieltjes transform. Apart from sample covariance matrices, linear eigenvalue statistics have been extensively studied for other types of matrices also: for Wigner matrices by Johansson \cite{johansson1998}, Sinai and Soshnikov \cite{soshnikov1998tracecentral}, Lytova and Pastur \cite{lytova2009central}; for tridiagonal matrices by Popescu \cite{Popescu}; for Toeplitz matrix by Chatterjee \cite{chatterjee2009fluctuations}, Liu, Sun and Wang \cite{liu2012fluctuations}, Li and Sun \cite{li_sun_2015}; for Hankel matrix by Liu, Sun and Wang \cite{liu2012fluctuations}, Kiran and Maurya \cite{kiran+maurya_hankel_22}; and for circulant type of matrices by Adhikari and Saha \cite{adhikari_saha2018}, Maurya and Saha \cite{m&s_RcSc_jmp}, Bose et al. \cite{a_m&s_circulaun2020}.
%After this work, several follow-up papers\textendash \cite{pan_CLT+signalinfere_08}, \cite{zheng_LSS+samcov_15}, \cite{najim_LSS+samcov_16} and \cite{zou_LSS+samcov+depen_20} have subsequently improved the CLT result of \cite{bai2004clt} and provided various applications in high-dimensional statistics. 
%These CLTs have successful applications in high-dimensional statistics by solving notably important problems in parameter estimation or hypothesis testing with high-dimensional data.
The linear eigenvalue statistics of random matrices have found applications in different areas, for example \textendash in parameter estimation and  hypothesis testing (see, \cite{bai_LRT+samcov_09}, \cite{wang_sphericaltest_13} and  \cite{najim_LSS+samcov_16}); for testing on regression coefficients (see \cite{bai_testing+regression_13} and \cite{jiang_testing+independence_13}) and for testing  the independence of two large dimensional variables (see \cite{jiang_testing+independence_13} and \cite{bonder_test+indepen+vec_19}).   

 The study of limiting spectral distribution (LSD) and linear eigenvalue statistics for matrices of the type $XX^\prime$ has received considerable attention from the random matrix theory community. The most well-known result is the Mar\v{c}enko–Pastur law, which states that if the entries of $X$ are i.i.d. random variables with mean zero, variance 1 and finite higher moments, then the limiting spectral distribution when $n \rightarrow \infty, p/n \rightarrow \lambda>0$, is the Mar\v{c}enko–Pastur distribution with parameter $\lambda$. 
 
 The distribution of eigenvalues of $XX^\prime$ for $n \times p$ matrices $X$ with dependence structure among its entries has also attracted attention in recent times. For the sample correlation matrix with $p/n \rightarrow \lambda$,  the limiting spectral distribution was shown to be the Mar\v{c}enko–Pastur distribution with parameter $\lambda$  by Jiang in \cite{Jiang_Correlation_2004}. 
 For $p/n \rightarrow \lambda>0$ the limiting spectral distribution for Spearman's rank correlation matrix was derived in \cite{Bai_Zhou_Spearman_LSD_2008}, and the linear eigenvalue statistics was studied in \cite{Bao+Pan_Spearman_2015}. In recent times, the matrices of type $XX^\prime$ has also found application in error propagation in neural networks \cite{Bahri_Neuralnetwork_2020}. For results on limiting spectral distribution of matrices of type $XX^\prime$ with dependence structure which arise from neural networks, see \cite{Pastur_Neural_Net-Gaussian_2020} and \cite{pastur_Neural_IID_2022}.

%For $\lambda \in [0, \infty)$, the limiting spectral distribution of $XX^\prime$ where $X$ is an $n \times p$ Toeplitz matrix or Hankel matrix and $p/n \rightarrow \lambda$ was studied in \cite{bose+gango+sen_XX'_10} for independent entries, and for \red{truncated independent} entries in \cite{bose+priyanka_XX'_22}. In this paper, we study the linear eigenvalue statistics of $XX'$ when $X$ is $n\times p$ Toeplitz or Hankel matrix with real, complex or time-dependent entries. 

For $\lambda \in [0, \infty)$, the limiting spectral distribution of $XX^\prime$ are also studied for patterned matrices $X$ when $p/n \rightarrow \lambda$. For $X$ as $n \times p$ Toeplitz matrix, Hankel, circulant and reverse circulant matrices with real independent entries of mean zero variance one and uniformly bounded higher moments, the LSD of $XX^\prime$ was studied in Bose et al. \cite{bose+gango+sen_XX'_10}. In a recent article by Bose and Sen \cite{bose+priyanka_XX'_22}, the results of \cite{bose+gango+sen_XX'_10} were generalized to weaker assumptions on the entries. In this paper, we study the linear eigenvalue statistics of $XX'$ when $X$ is $n\times p$ Toeplitz or Hankel matrix with real, complex or time-dependent entries. The techniques in this papers can also be applied to study the linear eigenvalue statistics of $XX^\prime$ for other patterned random matrices $X$ also.

%A sequence of variables $\left\{x_{i} ; i \in \Z \right\}$ will be called an input sequence if a matrix has the entries as $\left\{x_{i} ; i \in \Z \right\}$. 
%Let $\mathbb{Z}$ be the set of all integers and let $\mathbb{Z}_{+}$ denotes the set of all non-negative integers.  \red{Should we define $L_{n \times p}$ for defining $n \times p$ matrices. Neither It will not make sense to talk about $n \times p$ matrices}  Let
%$$L_{n,p}:\{1,2, \ldots, n\} \times \{1,2, \ldots, p\} \rightarrow \mathbb{Z} $$ 
%be a sequence of functions. If $L_{n+1}(i, j)=L_{n}(i, j)$ whenever $1 \leq i, j \leq n$, then the $\left\{L_{n}\right\}$ are said to be nested.
%For nested functions $L_n$, we write $L=L_n$ and consider $n \times n$ matrices of the form
%$$
%A_{n}=\left(x_{L(i, j)}\right) .
%$$
%The function $L$ is known as a link function and $A_n$ is known as the patterned matrix corresponding to the link function $L$. 
%\red{Similarly, we can define a sequence of functions $L_{n \times p}$ and we can  consider $n \times p$ matrices of the form $$ A_{n \times p}=\left(x_{L_{n \times p}(i, j)}\right).$$
%If $L_{n \times p}$ is symmetric, that is, $L_{n \times p}(i, j)=L_{n \times p}(j, i)$ for all $i, j$, then we will say the matrix $A_{n \times p}$ is \textit{symmetric}. Also for a complex input sequence $\left\{x_{i} ; i \in \Z \right\}$, we will say the matrix $A_{n \times p}$ is \textit{Hermitian} if $\bar{x}_{L_{n \times p}(i, j)}= x_{L_{n \times p}(j,i)}$.}
\vskip3pt 
\noindent \textbf{Toeplitz matrix:}
For a sequence of random variables $\{a_i\}_{i \in \Z}$, the corresponding $n\times p$ random Toeplitz matrix is defined as:
\begin{align*}
	T_{n \times p}=\left(\begin{array}{ccccc}
		a_{0} & a_{-1} & a_{-2} & \cdots & a_{1-p}\\
		a_{1} & a_{0} & a_{-1} & \cdots & a_{2-p}\\
		a_{2} & a_{1} & a_{0} & \cdots & a_{3-p}
		\\ \vdots &\vdots &\vdots & \cdots & \vdots\\
		a_{n-1}& a_{n-2}& a_{n-3}& \cdots & a_{n-p}
	\end{array}  \right).
\end{align*}
The $(i,j)$-th element of a Toeplitz matrix is $a_{i-j}$. Toeplitz matrices fall into different types depending on the relation between the $a_i$'s. We consider the following types:

\begin{enumerate}[(I)]
	\item If $\{a_i\}_{i \in \N}$ is a sequence of independent real-valued random variables and $a_{-i}=a_i$ for all $i \geq 1$, then we shall call $T_{n \times p}$ a \textit{symmetric Toeplitz matrix}. If $T_{n \times p}$ is not a symmetric Toeplitz matrix, we shall call $T_{n \times p}$, a \textit{non-symmetric Toeplitz matrix}.
	\item If $\{a_i\}_{i \in \N}$ is a sequence of independent complex-valued random variables and $a_{-i}=\bar{a}_i$ for all $i \geq 1$, then we shall call $T_{n \times p}$ a \textit{Hermitian Toeplitz matrix}.  	
\end{enumerate}  
%The symmetric Toeplitz matrix has its $(i,\;j)$-th element as $a_{|j-i|}$.
%\vskip2pt 
%\noindent \textbf{Hankel matrix:}  An $n\times p$  Hankel matrix $	H_{n \times p}$ is defined by its $(i,j)$-th element as $a_{i+j-2}$. Hankel matrices are closely related to Toeplitz matrices and Hankel matrices can be obtained from a Toeplitz matrix via row permutation. More specifically for the backward identity permutation matrix $P_{n \times n} = (\delta_{i-1, n-j} )_{i,j=1}^n$ and for any Toeplitz matrix $T_{n \times p}$, $P_{n \times n} T_{n \times p}$ is a Hankel matrix and conversely for any Hankel matrix $H_{n \times p}$, $P_{n \times n} H_{n \times p}$ is a  Toeplitz matrix. For a Hankel matrix of type $P_{n \times n} T_{n \times p}$, the $(i,j)$-th element is $a_{n+1-i-j}$ and it looks like the following:
%
%\begin{align*}
%	H_{n \times p}=\left(\begin{array}{ccccc}
%		a_{n-1}& a_{n-2}& a_{n-3}& \cdots & a_{n-p}\\
%		a_{n-2}& a_{n-3}& a_{n-4}& \cdots & a_{n-p-1}\\
%		a_{n-3}& a_{n-4}& a_{n-5}& \cdots & a_{n-p-2}
%		\\ \vdots &\vdots &\vdots & \cdots & \vdots\\
%		a_{0} & a_{-1} & a_{-2} & \cdots & a_{1-p}
%	\end{array}  \right).
%\end{align*}
% Note that $P_n^{-1} = P_n$.
% we obtain that $H_n=P_nT_n$, where $P_n=(\delta_{i-1,n-j})_{n\times n}$ is the {\it backward identity } matrix  of dimension $n$.
 
  For a Toeplitz matrix $T_{n\times p}$, we define the centralized and normalized linear eigenvalue statistics of $T_{n \times p} T'_{n \times p}$ as 
\begin{align} \label{eqn:wp_Hn_odd}
	w_k :=  \frac{1}{\sqrt{n}} \big[ \Tr(T_{n \times p} T'_{n \times p})^{k} - \E[\Tr(T_{n \times p} T'_{n \times p})^{k}] \big].
	%\eta_k :=  \frac{1}{\sqrt{n}} \big[ \Tr(H_{n \times p} H'_{p \times n})^{k} - \E[\Tr(H_{n \times p} H'_{p \times n})^{k}] \big].
\end{align}
For a sequence of random variables $\{x_i\}_{i \in I}$, we assume some moment conditions, which are required to state our main theorems.
\vskip2pt
\noindent \textbf{Assumption I.} Let $\{x_i\}_{i \in I}$ be a sequence of independent real-valued random variables with the following  moment conditions
$$\E[x_i]=0, \ \E[x_i^2]=1 \ \forall \ i \in I \text{ and} \ \sup _{i \in I} \E\big[\left|x_{i}\right|^{k}\big]=\alpha_{k}<\infty  \text { for } k \geq 3.$$
%\begin{equation}\label{eqn:condition}
%\E[x_i]=0, \ \E[x_i^2]=1 \ \forall \ i \in \Z \text{ and} \ \sup _{i \in \Z} \E\big[\left|x_{i}\right|^{k}\big]=\alpha_{k}<\infty  \text { for } k \geq 3.
%\end{equation}  
%\textcolor{red}{Should we assume that $a_j=a_{-j}$ for all $j=0,1,\ldots$ for the Toeplitz matrix?}

\noindent The following theorem provides the fluctuation of linear eigenvalue statistics of $T_{n \times p} T'_{n \times p}$.
\begin{theorem}\label{thm:LES_TT'}
Let $\lambda \in (0,\infty)$ and  $\{a_i\}_{i \in \N}$ be a sequence of random variables which satisfy Assumption I with $a_0 \equiv 0$. Suppose $T_{n \times p}$ is the $n \times p$ symmetric Toeplitz matrix with input entries $\{\frac{a_i}{\sqrt{n}}\}_{i \geq 0}$.
%, that is, $a_{-i}=a_i$ for all $i \in N$.
Then for every  $k\geq 1$, as $ n,p \tends \infty$ with $p/n \tends \lambda$,
\begin{equation*}
	w_k \stackrel{d}{\rightarrow} N_k,
\end{equation*}
	where $\{N_{k}; k\geq 1\}$ are zero mean Gaussian distributions with covariance structure as in (\ref{eqn:sigma_p,q_TT'-sym}).
Moreover, for a given polynomial $Q(x)=\sum_{j=0}^{k}c_{j} x^{j}$ with degree $k\geq 1$, 
\begin{equation*}
	w_Q:= \frac{1}{\sqrt{n}} \big[ \Tr Q(T_{n \times p}T'_{n \times p}) - \E[\Tr Q(T_{n \times p}T'_{n \times p})] \big] \stackrel{d}{\rightarrow} N(0,\sigma_Q^2),
\end{equation*}
 where $N(0,\sigma_Q^2)$ is the Gaussian distribution with mean zero and variance
 \begin{equation}\label{eqn:var Toeplitz symm}
 	\sigma_Q^2= \sum_{j_1,j_2=0}^{k}c_{j_1}c_{j_2} \Cov(N_{j_1},N_{j_2}).
 \end{equation}
\end{theorem}

\begin{remark}\label{rem:a0_neq0_sym}
In  Theorem \ref{thm:LES_TT'}, we considered $a_0\equiv 0$. If $a_0$ is a non-zero random variable, then the fluctuation of linear eigenvalue statistics may not be Gaussian. For the fluctuation result similar to Theorem \ref{thm:LES_TT'} for non-zero $a_0$, see Theorem \ref{thm:LES_TT'_a0}.
\end{remark}

The study of matrix-valued stochastic processes is also an area of considerable interest in random matrix theory. Among matrix-valued stochastic processes, of particular interest is the case where the matrix entries are Brownian motions. Next, we consider time-dependent symmetric Toeplitz matrix $T_{n\times p}(t)$ with entries $\{ \frac{b_m(t)}{\sqrt{n}};t\geq 0\}_{m\geq 0}$, where $\{b_m(t);t\geq 0\}_{m \geq 1}$ are independent standard Brownian motions and $b_0(t) \equiv 0$. We study the joint fluctuation and tightness of the time-dependent linear eigenvalue statistics for  $T_{n \times p}(t) T'_{n \times p}(t)$ 
with polynomial test functions. 
%For $T_{n \times p}(t) T'_{n \times p}(t)$ with $\phi(x)=x^{k}$, $k\geq 1$, observe that the linear eigenvalue statistics is $\Tr\big\{ T_{n \times p}(t) T'_{n \times p}(t)\big\}^{k}$. 
For $k \geq 1$, we define  the centralized and normalized linear eigenvalue statistics as 
%and centre $\Tr(RC_n(t))^{2p}$, to study its fluctuation, and define 
\begin{equation}\label{eqn:w_p(t)_Tp}
	w_{k}(t) := \frac{1}{\sqrt{n}} \bigl[ \Tr\big\{ T_{n \times p}(t) T'_{n \times p}(t)\big\}^{k} - \E[\Tr\big\{ T_{n \times p}(t) T'_{n \times p}(t)\big\}^{k}]\bigr]. 
\end{equation} 

%Similarly, for $H_{n \times p}(t) H'_{p \times n}(t)$ with $\phi(x)=x^{p}$, $p\geq 1$, let
%\begin{equation}\label{eqn:w_p(t)_Hp}
%	\eta_k(t) := \frac{1}{\sqrt{n}} \bigl[ \Tr\big\{ H_{n \times p}(t) H'_{p \times n}(t)\big\}^{k} - \E[\Tr\big\{ H_{n \times p}(t) H'_{p \times n}(t)\big\}^{k}]\bigr].  
%\end{equation}
\noindent The following is our main result for time-dependent entries.

\begin{theorem}\label{thm:LES_Tp} 
	 Suppose $ k\geq 1$ and $\lambda \in (0,\infty)$. Then as $ n,p \tends \infty$ with $p/n \tends \lambda$,
		\begin{equation*}
		\{ w_{k}(t) ; t \geq 0\} \stackrel{\mathcal D}{\rightarrow} \{N_{k}(t) ; t \geq 0\},
			\end{equation*}
		where $\{N_{k}(t);t\geq 0, k\geq 1\}$ are mean zero Gaussian processes with covariance structure as in (\ref{eqn:covTT'(t)}).
	%	where 	$\stackrel{\mathcal D}{\rightarrow}$ denotes the process convergence as in Definition \ref{def:processconvergence}.
\end{theorem}
%\begin{theorem}\label{thm:LES_Hp} 
% Suppose $ k\geq 1 $ and $y \in (0,\infty)$. Then as $ n,p \tends \infty$ with $p/n \tends y$,
%\begin{equation*}
%		\{ \eta_k(t) ; t \geq 0\} \stackrel{\mathcal D}{\rightarrow} \{N_k(t) ; t \geq 0\},
%\end{equation*}
%	where $\{N_k(t);t\geq 0, k\geq 1\}$ are zero mean Gaussian processes with  covariance structure as in \eqref{eqn:cov_Hp}. 
%\end{theorem
	Note that $\stackrel{d}{\rightarrow}$ is used to denote the weak convergence of random variables and $\stackrel{\mathcal D}{\rightarrow}$ is used to denote the process or functional convergence of random processes.

\begin{remark}\label{rem:a_0(t)neq0_sym}
We considered $b_0(t)\equiv 0$ in  Theorem \ref{thm:LES_Tp}. For non-zero continuous stochastic process $b_0(t)$, the limit may not be a Gaussian process. For details regarding the case $b_0(t)\neq 0$, see Theorem \ref{thm:LES_Tp_b0}.
\end{remark}
We also study the linear eigenvalue statistics of $T_{n\times p} T_{n\times p}'$ for other types of Toeplitz matrices. In Section \ref{sec:non-sym_Tnp}, we consider non-symmetric Toeplitz matrix $T_{n\times p}$. In Theorem \ref{thm:LES_TT'non-sym}, we establish results for the linear eigenvalue statistics of $T_{n\times p}T'_{n\times p}$ when the entries of $T_{n\times p}$ are independent, and in Theorem \ref{thm:LES_Tp(t)_non-sym}, we study the process convergence of time-dependent linear eigenvalue statistics of $T_{n\times p}(t)T'_{n\times p}(t)$ when the entries are independent Brownian motions. In Section \ref{sec:her_Tnp}, we consider  Hermitian Toeplitz matrix $T_{n\times p}$ and study the linear eigenvalue statistics of $T_{n\times p}T^*_{n\times p}$ for both, independent (see  Theorem \ref{thm:LES_TT'her}) and Brownian motion entries (see Theorem \ref{thm:LES_Tp(t)_her}). From the proofs of Theorem \ref{thm:LES_TT'} and Theorem \ref{thm:LES_Tp}, we see that the results and idea of proofs for  symmetric Toeplitz matrix is more intricate compared to the non-symmetric Toeplitz matrix and Hermitian Toeplitz matrix. Due to this, our main focus is on symmetric Toeplitz matrix.

\vskip4pt 
\noindent \textbf{Hankel matrix:} Given a sequence of random variables $\{a_i\}_{i \in \N}$, the corresponding $n\times p$  Hankel matrix is defined as $H_{n \times p}=\left(a_{i+j-1}\right)_{i,j=1}^{n,p}$. Hankel matrix is closely related to Toeplitz matrix and a Hankel matrix can be alternatively defined as $H_{n \times p}=P_{n \times n} T_{n \times p}$, where $T_{n \times p}$ is the non-symmetric Toeplitz matrix with input sequence $\{a_i\}_{i \in \Z}$ and $P_{n\times n}$ is the backward identity matrix  of dimension $n$. A Hankel matrix of type $P_{n \times n} T_{n \times p}$, is given by:
\begin{align*}
	H_{n \times p}=\left(\begin{array}{ccccc}
		a_{n-1}& a_{n-2}& a_{n-3}& \cdots & a_{n-p}\\
		a_{n-2}& a_{n-3}& a_{n-4}& \cdots & a_{n-p-1}\\
		a_{n-3}& a_{n-4}& a_{n-5}& \cdots & a_{n-p-2}
		\\ \vdots &\vdots &\vdots & \cdots & \vdots\\
		a_{0} & a_{-1} & a_{-2} & \cdots & a_{1-p}
	\end{array}  \right).
\end{align*}

\noindent The following remark provides the fluctuation behavior of linear eigenvalue statistics of $H_{n\times p}H'_{n\times p}$.
\begin{remark}
	Note that if we consider Hankel matrix  $H_{n \times p}$ of the form  $P_{n\times n}T_{n \times p}$, where $T_{n \times p}$ is an $n \times p$ non-symmetric Toeplitz matrix and $P_{n\times n}=(\delta_{i-1,n-j})_{n\times n}$ is the {\it backward identity } matrix  of dimension $n$, then 
	$$H_{n \times p} H'_{n \times p}= P_{n\times n} T_{n \times p} T'_{n \times p}  P_{n\times n}.$$
	Since $P_{n\times n}= (P_{n\times n})^{-1}$, we have that $H_{n \times p} H'_{n \times p}$ and $T_{n \times p} T'_{n \times p}$ have same eigenvalues and hence behaviour of the fluctuation of linear eigenvalue statistics of $H_{n \times p} H'_{n \times p}$ is same as that of $T_{n \times p} T'_{n \times p}$ obtained in Theorem \ref{thm:LES_TT'non-sym}.
\end{remark}
Schatten $r$-norms are an important class of matrix norms that includes trace norm and Frobenius norm as special cases. For a compact operator $A$ between two Hilbert spaces, the Schatten $r$-norm of $A$ is defined as the $\ell_r$ norm of singular values of $A$, provided it is finite. The space of all $r$-summable compact operators on a Hibert space under this norm is a Banach space and furthermore, Schatten $r$-norms are invariant under multiplication by unitary operators. Schatten $r$-norms naturally arises in the study of non-self adjoint operators and symmetrically normed operator ideals. For detailed discussions on these connections, see \cite{Gohberg_Non-Selfadjoint_1969}.
As a result of its connection to singular values, Schatten norms are suitable candidates for applications in low-rank matrix approximation \cite{Nie_Huang_Ding_2021} and image reconstruction \cite{Xie+Gu_Schatten_norm2016},\cite{Stamatios_Image_Schatten_2013}.

Our studies on linear eigenvalue statistics gives the asymptotic behaviour of Schatten $r$-norms of $n\times p$ Toeplitz matrices for large values of $n$ and $p/n \rightarrow \lambda>0$. The results obtained in this paper are analogous to the results on fluctuations of spectral norms of square Toeplitz matrices and Hankel matrices in \cite{bose+Sen_spec+norm_07}.

Now we briefly outline the rest of the manuscript. In Section \ref{sec:pre_Hankelodd}, we introduce some combinatorial objects and derive some results which are needed for the proofs of theorems. In Section \ref{sec:symToep_ind} and Section \ref{sec:T_np_process},  we consider symmetric Toeplitz matrix  and prove Theorem \ref{thm:LES_TT'} and Theorem \ref{thm:LES_Tp}, respectively. In Section \ref{sec:non-sym_Tnp}, we consider  non-symmetric Toeplitz matrix and study its linear eigenvalue statistics, and in Section \ref{sec:her_Tnp} we  study linear eigenvalue statistics of Hermitian Toeplitz matrices. Finally in Section \ref{sec:schatten-norm}, we study fluctuation of Schatten $r$-norm for different types of Toeplitz matrices.
%Finally, in Section \ref{sec:wp_process_Hnodd} we prove Theorem \ref{thm:process_wp_hnodd} by using moment method and some results on process convergence. 
\section{\large \bf \textbf{Preliminaries}} \label{sec:pre_Hankelodd}

We first introduce certain partitions, concept of matching and some integrals associated with them. Later, we derive the trace formula for $XX'$ when $X$ is a Toeplitz matrix.

\begin{definition}\label{defn:partition}
   Consider the set $[n]=\{ 1,2 ,\ldots , n\}$.  We call $\pi=\{V_{1},\ldots,V_{r}\}$ a partition of $[n]$ if the
   blocks  $V_{j} \,(1\leq j \leq r)$ are pairwise disjoint, non-empty
   subsets of $[n]$ such that $[n]=V_{1}\cup\cdots\cup V_{r}$. The set of all partitions of $[n]$ is denoted by
   $\mathcal{P}(n)$. The number of blocks in a partition $\pi$ is denoted by $|\pi|$, and the number of
   elements of a block $V_{j}$ is denoted by $|V_{j}|$. If $i,j$ belong to the same block of $\pi$, we write $i \sim_{\pi} j$. A partition $\pi$ of $[n]$ is called a \textit{pair-partition} if each block of $\pi$ has exactly two elements. The set of all pair-partitions of $[n]$ is denoted by $\PP_2(n)$. Clearly, $\PP_2(n)= \emptyset$ for odd $n$. 
\end{definition}

For a partition $\pi$ of $[n]$, we define a canonical ordering of its blocks by arranging the blocks in the increasing order of their smallest elements. For a partition $\pi$ with blocks arranged as $B_1, B_2,\ldots, B_k$, we define a surjective function, also denoted by $\pi: [n ]\rightarrow [k]$, given by 
\begin{equation}\label{eqn: pi map}
\pi (i)=j, \  \ \mbox{if $i \in B_j$}.
\end{equation}

In order to calculate the moments of the limiting distribution, and the limiting  covariance for the linear eigenvalue statistics, we first review some basic combinatorical
concepts associated with them. Later we also define some integrals associated with different partitions.

\begin{definition} \label{def:P_2,4}
	Let $k_1,k_2$ be positive integers.
	 \begin{enumerate}[(i)]
%	 	\item Without loss of generality, we assume that $V_{1},\ldots,V_{r}$
%	 	have been arranged such that $s_{1}<s_{2}<\cdots<s_{r}$, where
%	 	$s_{j}$ is the smallest number of $V_{j}$. Therefore we can define
%	 	the projection $\pi (i)=j$ if $i$ belongs to the block $V_{j}$;
%	 	furthermore for two elements $p,q$ of $[n]$ we write
%	 	$p\thicksim_{\pi} q$ if $\pi (p)=\pi (q)$.
	
	 	\item Suppose $k_1+k_2$ is even. Consider the set of all pair-partitions $\pi \in \mathcal{P}_{2}(k_1+k_2)$ such that there exist $1\leq i
	 	\leq k_1<j \leq k_1+k_2$ such that $i\thicksim_{\pi} j$. In this case, we say that there is \textit{cross-matching} in $\pi$ and the set of all pair-partitions  $\pi \in \mathcal{P}_{2}(k_1+k_2)$ with cross-matching is denoted by $\mathcal{P}_{2}(k_1,k_2)$.
	 	
	 	\item When $k_1$ and $k_2$ are both  even, we denote a subset of
	 	$\mathcal{P}(k_1+k_2)$ by $\mathcal{P}_{2,4}(k_1,k_2)$, which consists
	 	of all  partitions $\pi=\{V_1,V_2,\ldots,V_{r}\}$ satisfying
	 	\begin{itemize}
	 		\item [(a)] $|V_{i}|=4$ for some
	 		$i$ and $|V_{j}|=2$, for all  $1\leq j \leq r, j\neq i$.
	 		
	 		\item [(b)]  $V_{j}\subseteq \{1,2, \ldots, k_1\} \, \mathrm{or}\, \{k_1+1,
	 		k_1+2, \ldots, k_1+k_2\}$ for all $1\leq j \leq r, j\neq i$.
	 		
	 		\item [(c)] two elements of $V_{i}$ come from $\{1,2, \ldots, k_1\}$ and the
	 		other two come from
	 		$\{k_1+1, p+2, \ldots, k_1+k_2\}$. 
	 	\end{itemize}
 		For other cases of $k_1$ and $k_2$, we  assume $\mathcal{P}_{2,4}(k_1,k_2)$ is an empty set.
	 \end{enumerate}

	% When $q=0$, we mean that $\mathcal{P}_{2}(p,0)$ denotes
	%$\mathcal{P}_{2}(p)$. In this section we assume $p+q$ is even %unless
	%we mention it
\end{definition}

\begin{definition}\label{def:SP+DP_partation}
	For a partition $\pi$ belonging to $ \mathcal{P}_2(k_1,k_2)$ or $\pi \in \mathcal{P}_{2,4}(k_1,k_2)$, we define the following types of blocks. 
	\begin{enumerate}[(i)]
		\item A block $\{r,s\}$ is called a \textit{same-parity block} if $r$ and $s$ have same parity.
		\item A block $\{r,s\}$ is called a \textit{different-parity block} if $r$ and $s$ have different parity. 
		\item A block $\{r_1,s_1,r_2,s_2\}$ with $r_1<s_1<r_2<s_2$ is called a \textit{different-parity block} if both $r_1+s_1$ and $r_2+s_2$ are odd. Otherwise the block is called a \textit{same-parity block}.
	\end{enumerate}
 The set of all same-parity blocks of $\pi$ is denoted by $SP(\pi)$ and the set of all different-parity blocks of $\pi$ is denoted by $DP(\pi)$.
%	and for a block $B_i$ such that $|B_i|=4$ with $B_i \cap [k_1]=\{r_1,s_1\}$ and $B_i \cap ([k_1] \setminus [k_2])=\{r_2,s_2\}$, we say 
\end{definition}

\begin{definition} \label{def:SP_2}
	A partition $\pi \in \mathcal{P}_2(2k_1)$ or $\pi \in \mathcal{P}_{2,4}(2k_1,2k_2)$ is called a \textit{different-parity partition}, if all blocks of $\pi$ are different-parity blocks. We use the notations, $\mathcal{DP}_{2}(2k_1)$, $\mathcal{DP}_{2}(2k_1,2k_2)$ and $\mathcal{DP}_{2,4}(2k_1,2k_2)$ to denote the set of all different-parity partitions in $\mathcal{P}_2(2k_1)$, $\mathcal{P}_2(2k_1,2k_2)$ and $\mathcal{P}_{2,4}(2k_1,2k_2)$, respectively.
%	\begin{enumerate}[(i)]
%		\item The set of all different-parity partitions in $\mathcal{P}_2(2k_1)$ shall be denoted by $\mathcal{DP}_{2}(2k_1)$.
%		\item The set of all different-parity partitions in $\mathcal{P}_2(2k_1,2k_2)$ shall be denoted by $\mathcal{DP}_{2}(2k_1,2k_2)$.
%		\item The set of all different-parity partitions in $\mathcal{P}_{2,4}(2k_1,2k_2)$ shall be denoted by $\mathcal{DP}_{2,4}(2k_1,2k_2)$.
%	\end{enumerate}
\end{definition}
Now  we define several types of definite integrals associated with
partitions in $\mathcal{P}_{2}(p,q)$ and  $\mathcal{P}_{2,4}(p,q)$.
For the reader's  convenience, we suggest to omit them for the
moment  and refer to them when they are needed in Section
\ref{sec:proof_thm_TT'} and Section \ref{sec:non-sym_Tnp}.
Consider a pair-partition $\pi \in \mathcal{P}_{2}(2k_1+2k_2)$. Note that $\pi$ can be expressed as $\pi=\{r_1,s_1\},\{r_2,s_2\},\ldots ,\{r_{k_1+k_2},s_{k_1+k_2}\}$ where for each $1 \leq i \leq k_1+k_2$, $r_i<s_i$. We consider the map $\epsilon_{\pi}:[2(k_1+k_2)] \rightarrow \{ \pm 1\}$ given by
%\begin{align} \label{eqn: epsilon map}
%\epsilon_{\pi}(r_i)=
%	\begin{cases} 1,\ \
%		i \ \text{is the smallest number of}\  \pi^{-1}(\pi(i));\\
%		(-1)^{r+s+1},\ \ \text{otherwise and if }\pi^{-1}(\pi(i))=\{r,s\}. 
%	\end{cases}
%\end{align}
\begin{align} \label{eqn: epsilon map}
	\epsilon_{\pi}(r_i) = 1  \text{ and }
	\epsilon_{\pi}(s_i) =(-1)^{1+r_i+s_i} \text{ for } 1 \leq i \leq k_1+k_2.
\end{align}
 For every pair partition $\pi \in \mathcal{P}_{2}(2(k_1+k_2))$, we
construct a projective relation between two groups of unknowns
${y_{1},\ldots,y_{2k_1+2k_2}}$ and ${x_{1},\ldots,x_{k_1+k_2}}$ as
follows:
\begin{equation}\label{eqn:relation y_i,x_i}
\epsilon_{\pi}(i)\,y_{i}=\epsilon_{\pi}(j)\,y_{j}=x_{\pi(i)}
\end{equation}
 whenever $i\thicksim_{\pi} j$. Note that $r,s$ have same parity if and only if $\epsilon_\pi(r)+\epsilon_\pi(s)=0$. Thus for all choices of unknowns ${x_{1},\ldots,x_{k_1+k_2}}$, the equation $\sum_{j=1}^{2k_1+2k_2}(-1)^jy_{j}=0$ is automatically satisfied.

%For $x_{0}, y_{0}, \lambda \in [0,1]$ and ${x_{1},\ldots,x_{k_1+k_2}\in	[-1, 1]}$, we define two kinds of integrals with Type I by

For $\pi \in \mathcal{P}_2(2k_1,2k_2)$ and $\lambda \in \R^+$, we define the following integrals:
%\prod_{j=1}^{p+q}\chi_{[0,1]}(x_{0}+b\sum_{i=1}^{j}y_{i})
\begin{align}	
	 f_{I}^{-}(\pi)=
	& \int_{[0,1]^2 \times [-m,m]^{|SP(\pi)|} \times[-\lambda,1]^{|DP(\pi)|}} \delta\left(\sum_{i=1}^{2k_1} (-1)^i y_{i}\right)
	\prod_{j=1}^{k_1}\chi_{[0,\lambda]}(x_{0}+\sum_{i=1}^{2j-1} y_{i}) \,\,\chi_{[0,1]}(x_{0}+\sum_{i=1}^{2j} y_{i}) \nonumber \\
	& \quad \times \prod_{j'=k_1+1}^{k_1+k_2} \chi_{[0,\lambda]}(y_{0}+\sum_{i=2k_1+1}^{2j'-1} (-1)^i y_{i})  \,\,\chi_{[0,1]}(y_{0}+\sum_{i=2k_1+1}^{2j'} (-1)^i y_{i}) \,d\,y_{0} 
	\prod_{l=0}^{k_1+k_2} d\, x_{l}, \label{eqn:f_Ipi-_T} 
\end{align} 
%and
\begin{align}
	 f_{I}^{+}(\pi)=
	& \int_{[0,1]^2 \times [-m,m]^{|SP(\pi)|} \times[-\lambda,1]^{|DP(\pi)|}} \delta\left(\sum_{i=1}^{2k_1} (-1)^i y_{i}\right)
	\prod_{j=1}^{k_1}\chi_{[0,\lambda]}(x_{0}+\sum_{i=1}^{2j-1} y_{i}) \,\,\chi_{[0,1]}(x_{0}+\sum_{i=1}^{2j} y_{i}) \nonumber \\
& \quad \times \prod_{j'=k_1+1}^{k_1+k_2} \chi_{[0,\lambda]}(y_{0}-\sum_{i=2k_1+1}^{2j'-1} (-1)^i y_{i})  \,\,\chi_{[0,1]}(y_{0}-\sum_{i=2k_1+1}^{2j'} (-1)^i y_{i}) \,d\,y_{0} 
\prod_{l=0}^{k_1+k_2} d\, x_{l}, \label{eqn:f_Ipi+_T}
\end{align} 
where $m=\min\{\lambda,1\}$; $SP(\pi), DP(\pi)$ are as in Definition \ref{def:SP+DP_partation}; $\delta_x$ is
the Dirac function concentrated at $x$ and $\chi$ is the indicator  function. In (\ref{eqn:f_Ipi-_T}) and (\ref{eqn:f_Ipi+_T}), $x_i$ varies in the interval $[-m,m]$ if the block $B_i$ is a same-parity block and varies in $[-\lambda,1]$ if the block $B_i$ is a different-parity block. %\red{Note that for $\lambda=0$, $f_{I}^{-}(\pi)= f_{I}^{+}(\pi)=0$.}

Now we define the integrals associated with partitions in $\PP_{2,4}(2k_1,2k_2)$. For $\pi=\{V_{1},\ldots,V_{k_1+k_2-1}\}\in
\mathcal{P}_{2,4}(2k_1,2k_2)$ (denoting the block with four elements by
$V_{i}$), we define the map $\tau_{\pi}:[2(k_1+k+2)] \rightarrow \{\pm 1\}$ in the following way: for $V_j=\{r,s\}$ with $r<s$, we define
	\begin{eqnarray}\label{eqn:tau map}
		 \tau_{\pi}(r)=1 \text{ and } 
		 \tau_{\pi}(s)=(-1)^{r+s+1}.
	\end{eqnarray}
	 While for $V_i=\{r_1,s_1,r_2,s_2\}$ with $r_1<s_1<r_2<s_2$, we define
	\begin{eqnarray*}
		 \tau_{\pi}(k)=
		\begin{cases} 1 \ \
			& \text{ if } k =r_1 \ \text{or } r_2,\\
			(-1)^{r_i+s_i+1} \ \ & \text{ if } k=s_1 \text{ or } s_2.
		\end{cases}
	\end{eqnarray*}
	For every partition $\pi \in \mathcal{P}_{2,4}(2k_1,2k_2)$, we
	construct a projective relation between two groups of unknowns
${y_{1},\ldots,y_{2k_1+2k_2}}$ and ${x_{1},\ldots,x_{k_1+k_2-1}}$  as
	follows:
	\begin{equation*}
		\tau_{\pi}(i)\,y_{i}=\tau_{\pi}(j)\,y_{j}=x_{\pi(i)}
	\end{equation*}
	 whenever $i\thicksim_{\pi} j$. Then for every choice of $x_1,x_2, \ldots , x_{k_1+k_2-1}, $ the equations  $\sum_{j=1}^{2k_1}(-1)^jy_{j}=0$ and $\sum_{j=2k_1+1}^{2k_1+2k_2}(-1)^jy_{j}=0$ are always satisfied. For a partition $\pi \in \PP_{2,4}(2k_1,2k_2)$ and $\lambda \in \R^+$,  we define the following integrals:
	\begin{align} 
		f_{II}^{-}(\pi) &=
		 \int_{[0,1]^{2}\times
			[-m,m]^{|SP(\pi)|} \times[-\lambda,1]^{|DP(\pi)|}} 
			\prod_{j=1}^{k_1}\chi_{[0,\lambda]}(x_{0}+\sum_{i=1}^{2j-1} y_{i}) \,\,\chi_{[0,1]}(x_{0}+\sum_{i=1}^{2j} y_{i}) \nonumber \\
		& \quad \times \prod_{j'=k_1+1}^{k_1+k_2} \chi_{[0,\lambda]}(y_{0}+\sum_{i=2k_1+1}^{2j'-1} (-1)^i y_{i})  \,\,\chi_{[0,1]}(y_{0}+\sum_{i=2k_1+1}^{2j'} (-1)^i y_{i}) \,d\,y_{0}
		\prod_{l=0}^{k_1+k_2-1} d\, x_{l}, \label{eqn:f_II-pi_T}  \\
%	\end{align} 
%and
%	\begin{align}
		 f_{II}^{+}(\pi) &=
		 \int_{[0,1]^{2}\times
		 [-m,m]^{|SP(\pi)|} \times[-\lambda,1]^{|DP(\pi)|}} 
		 	\prod_{j=1}^{k_1}\chi_{[0,\lambda]}(x_{0}+\sum_{i=1}^{2j-1} y_{i}) \,\,\chi_{[0,1]}(x_{0}+\sum_{i=1}^{2j} y_{i}) \nonumber \\
		 & \quad \times \prod_{j'=k_1+1}^{k_1+k_2} \chi_{[0,\lambda]}(y_{0}-\sum_{i=2k_1+1}^{2j'-1} (-1)^i y_{i})  \,\,\chi_{[0,1]}(y_{0}-\sum_{i=2k_1+1}^{2j'} (-1)^i y_{i})
		 \,d\,y_{0} \prod_{l=0}^{k_1+k_2-1} d\, x_{l}. \label{eqn:f_II+pi_T}
	\end{align}
  where $m=\min\{\lambda,1\}$ and $SP(\pi), DP(\pi)$ are as in Definition \ref{def:SP+DP_partation}.

Note that for $\pi \in \mathcal{DP}_2(2k_1,2k_2))$, the function $\epsilon_{\pi}$ in (\ref{eqn: epsilon map}) is the constant function taking value 1. Thus the set of equations (\ref{eqn:relation y_i,x_i}) become $y_i=x_{\pi(i)}$ for all $i$. Furthermore, note that in this case $|SP(\pi)|=0$ and $|DP(\pi)|=k_1+k_2$ and therefore the domain of integration for $f_I^-(\pi)$ and $f_{I}^+(\pi)$ is $[0,1]^{2} \times[-\lambda,1]^{k_1+k_2}$. Similar argument implies that for $f_{II}^-(\pi)$ and $f_{II}^+(\pi)$, the domain of integration is $[0,1]^{2} \times[-\lambda,1]^{k_1+k_2-1}$.
                                                                          
 Now we derive the trace formula for the product of the matrices of form $T_{n\times p} T_{n \times p}'$, which is required to prove Theorem \ref{thm:LES_TT'} and Theorem \ref{thm:LES_Tp}.
\begin{lemma} \label{lem:trac_TT^t}
	Suppose  $T^{(r)}_{n\times p}$ are $n\times p$ Toeplitz matrices with input sequence $\{ a^{r}_i\}_{ i \in \mathbb{Z} }$ for $r=1,2, \ldots, k$. Let $X^{(r)}=T^{(r)}_{n\times p} (T^{(r)}_{n \times p})'$. Then
	\begin{align} \label{def:trace_Tn}
		\Tr( X^{(k)} \cdots X^{(1)})
		= \sum_{i=1}^{n} \sum_{j_{1}, \ldots, j_{2k}=-(p-1)}^{n-1} \prod_{r=1}^{k} a^{(r)}_{j_{2r-1}} a^{(r)}_{j_{2r}}  I(i,J,p,n)\delta_{0} (\sum_{q=1}^{2k}(-1)^{q} j_{q}),
%= \sum_{i=1}^{n} \sum_{A_k} \prod_{r=1}^{2k} a_{j_{r}}   I_J(i,p) I_J(i,n),
	\end{align}
	where $J=(j_1,j_2,\ldots,j_{2k})$, $\delta_x$ is the Dirac delta function at $x$ and
	\begin{align} \label{eqn:I(i,j,p,n)}
		I(i,J,p,n) = \prod_{s=1}^{k} \chi_{[1,n]}(i+\sum_{\ell=1}^{2s-1} (-1)^\ell j_\ell)  \chi_{[1,p]}(i+\sum_{\ell=1}^{2s} (-1)^\ell j_\ell ).
	%	I(i,J,p,n) = \prod_{s=1, \ s \mbox{ odd}}^{2k} \chi_{[1,p]}(i+\sum_{\ell=1}^{s} (-1)^\ell j_\ell ) \times \prod_{s=1, \ s \mbox{ even}}^{2k} \chi_{[1,n]}(i+\sum_{\ell=1}^{s} (-1)^\ell j_\ell ),
	\end{align}
%	\begin{align} \label{eqn:I(i,j,p,n)}
%	I_J(i,p) = \prod_{s=1, \ s \mbox{ odd}}^{2k} \chi_{[1,p]}(i+\sum_{\ell=1}^{s} (-1)^\ell j_\ell ), \
%	I_J(i,n) = \prod_{s=1, \ s \mbox{ even}}^{2k} \chi_{[1,n]}(i+\sum_{\ell=1}^{s} (-1)^\ell j_\ell ),
%\end{align}
%with $\chi$ as the indicator function and
\end{lemma}
%\begin{remark}
%	Note that the trace formula of Lemma \ref{lem:trac_TT^t} can also be written as
%	\begin{align} \label{def:trace_Tn_A}
%		\Tr( X^{(k)} \cdots X^{(1)})
%		= \sum_{i=1}^{n} \sum_{A_k}  \prod_{r=1}^{k} a^{(r)}_{j_{2r-1}} a^{(r)}_{j_{2r}}  I(i,J,p,n),
%	\end{align}
%where $I(i,J,p,n)$ is as in (\ref{eqn:I(i,j,p,n)}) and 
%\begin{align}\label{eqn:A_k}
%	A_k= \big\{(j_1,\ldots, j_{2k}) \in \{-(p-1),\ldots, -1, 0, 1,\ldots, (n-1)\}^{2k} : \sum_{q=1}^{2k}(-1)^{q} j_{q}=0 \big\}.
%\end{align}
%\end{remark}
\begin{proof}[Proof of Lemma \ref{lem:trac_TT^t}]
	First note from the structure of Toeplitz matrix that 
	\begin{align*}% \label{eqn:T+F+G}
		T^{(1)}_{n \times p}= \sum_{j=0}^{n-1} a^{(1)}_{j} F_{n \times p}{(j)} + \sum_{j=1}^{p-1} a^{(1)}_{-j} G_{n \times p}{(j)},
	\end{align*}
where for $j=0,1, \ldots, (n-1)$,
\begin{align*} %\label{eqn:Fe_i}
	 F_{n \times p}{(j)} e_i &= 
		\begin{cases}
		e_{i+j} &  \mbox{if } 1 \leq (i+j) \leq n,\\
		0 &\mbox{ otherwise},
	\end{cases} \nonumber \\
&= \chi_{[1, n]}(i+j) e_{i+j},
\end{align*}
	and for $j=1, \ldots, (p-1)$,
	\begin{align*} %\label{eqn:Ge_i}
		G_{n \times p}{(j)} e_i &= 
		\begin{cases}
			e_{i-j} &  \mbox{if } 1 \leq (i-j) \leq n,\\
			0 &\mbox{ otherwise},
		\end{cases}  \nonumber \\
	&= \chi_{[1, n]}(i-j) e_{i-j}.
	\end{align*}
Thus  for $i=1,2, \ldots,n$, we have
\begin{align*} %\label{eqn:Te_i}
	(T^{(1)}_{n \times p})e_i= \sum_{j=0}^{n-1} a^{(1)}_{j}\chi_{[1, n]}(i+j) e_{i+j} + \sum_{j=1}^{p-1} a^{(1)}_{-j} \chi_{[1, n]}(i-j) e_{i-j} = \sum_{j=-(p-1)}^{n-1} a^{(1)}_{j}\chi_{[1, n]}(i+j) e_{i+j}.
\end{align*}
Similarly, we can also show that
\begin{align*}% \label{eqn:T^te_i}
	(T^{(1)}_{n \times p})' e_i= \sum_{j=0}^{p-1} a^{(1)}_{-j}\chi_{[1, p]}(i+j) e_{i+j} + \sum_{j=1}^{n-1} a^{(1)}_{j} \chi_{[1, p]}(i-j) e_{i-j} = \sum_{j=-(p-1)}^{n-1} a^{(1)}_{j}\chi_{[1, p]}(i-j) e_{i-j}.
\end{align*}
Now from the above two expressions, for $i=1,2, \ldots,n$, we have
\begin{align*}% \label{eqn:TT^te_i}
	\big(T^{(1)}_{n \times p} (T^{(1)}_{n \times p})' \big)e_i &= T^{(1)}_{n \times p} \Big[\sum_{j_1=-(p-1)}^{n-1} a^{(1)}_{j_1}\chi_{[1, p]}(i-j_1) e_{i-j_1}\Big] \nonumber \\
	&= \sum_{j_1=-(p-1)}^{n-1} a^{(1)}_{j_1}\chi_{[1, p]}(i-j_1) \Big[ \sum_{j_2=-(p-1)}^{n-1} a^{(1)}_{j_2}\chi_{[1,n]}(i-j_1+j_2) e_{i-j_1+j_2}\Big] \nonumber \\
	&= \sum_{j_1, j_2=-(p-1)}^{n-1} a^{(1)}_{j_1}  a^{(1)}_{j_2} \chi_{[1, p]}(i-j_1) \chi_{[1,n]}(i-j_1+j_2) e_{i-j_1+j_2}.
\end{align*}
Similarly, for $i=1,2, \ldots,n$ and $X^{(r)}=T^{(r)}_{n\times p} (T^{(r)}_{n \times p})'$, we have
\begin{align*} %\label{eqn:(TT^t)^2e_i}
	( X^{(2)}_n X^{(1)}_{n})e_i &=
	 \sum_{j_1, j_2=-(p-1)}^{n-1} a^{(1)}_{j_1}  a^{(1)}_{j_2} \chi_{[1, p]}(i-j_1) \chi_{[1,n]}(i-j_1+j_2) 	( X^{(2)}_n) e_{i-j_1+j_2}  \nonumber \\ 
	 &= \sum_{j_1, \ldots, j_4=-(p-1)}^{n-1} \prod_{r=1}^{2}  a^{(r)}_{j_{2r-1}} a^{(r)}_{j_{2r}} \chi_{[1, p]}(i-j_1) \chi_{[1,n]}(i-j_1+j_2)  \nonumber \\ 
	& \qquad \times \chi_{[1,p]}(i-j_1+j_2-j_3) \chi_{[1,n]}(i-j_1+j_2-j_3+j_4) e_{i-j_1+j_2-j_3+j_4}.
\end{align*}
Continuing this process, we get
\begin{align} \label{eqn:(TT^t)^ke_i}
	( X^{(k)} \cdots X^{(1)})e_i &=
 \sum_{j_{1}, \ldots, j_{2k}=-(p-1)}^{n-1} \prod_{r=1}^{k} a^{(r)}_{j_{2r-1}} a^{(r)}_{j_{2r}}  \prod_{s=1}^{k} \chi_{[1,n]}(i+\sum_{\ell=1}^{2s-1} (-1)^\ell j_\ell)   \nonumber \\ 
	& \qquad  \times \chi_{[1,p]}(i+\sum_{\ell=1}^{2s} (-1)^\ell j_\ell ) e_{i-\sum_{q=1}^{2k}(-1)^{q} j_{q}}.
\end{align}
Finally, using the fact that $\Tr( X^{(k)} \cdots X^{(1)}) = 	\sum_{i=1}^{n} e'_i ( X^{(k)} \cdots X^{(1)})e_i$, we get  (\ref{def:trace_Tn}) from (\ref{eqn:(TT^t)^ke_i}). This completes the proof of Lemma \ref{lem:trac_TT^t}.
\end{proof}

\section{ \textbf{ \large \bf Symmetric Toeplitz with independent entries}} \label{sec:symToep_ind}  
In this section, we deal with $T_{n\times p} T'_{n \times p}$ when $T_{n\times p}$ is symmetric Toeplitz matrix with  independent entries. In Section \ref{subsec:momentseq_symTT'}, we compute the limiting moments for the trace of $\frac{1}{n}T_{n\times p} T'_{n \times p}$ and in Section \ref{sec:proof_thm_TT'}, we study the linear eigenvalue statistics of  $T_{n\times p} T'_{n \times p}$ when the diagonal entries are zero (Theorem \ref{thm:LES_TT'}). For non-zero diagonal entries, the linear eigenvalue statistics for  $T_{n\times p} T'_{n \times p}$ is studied in Section \ref{sec:proof_thm_TT'_nonzero}.

\subsection{\large \bf Limiting moment sequence} \label{subsec:momentseq_symTT'} In this section, we find the limiting moments of $T_{n \times p}T_{n \times p}^\prime$.
Note that the $r$-th moment of the \textit{empirical spectral distribution} of $T_{n\times p} T'_{n \times p}$ can be written as $\E\left[\frac{1}{n} \Tr( T_{n\times p} T'_{n \times p})^r\right]$. In \cite{bose+gango+sen_XX'_10} the limit of this quantities were calculated for establishing the limiting spectral distribution for $T_{n \times p}T_{n \times p}^\prime$, but a closed form of the sequence was not obtained. Here, our following theorem provides an explicit form of the limit in terms of an integral. Later in Section \ref{sec:proof_thm_TT'_nonzero}, this theorem will also be used to study  the linear eigenvalue statistics of  $T_{n\times p} T'_{n \times p}$ when the diagonal entries of $T_{n\times p}$ are non-zero and in Section \ref{sec:schatten-norm}, this theorem shall be used to find the asymptotic distribution of Schatten norm of $T_{n \times p}$.
%In this section, we compute the limiting moments for $w_p$, when $T_{n \times p}$ is the rectangular version of the symmetric Toeplitz matrix.
\begin{theorem}\label{thm: Moment TT^t}
%Suppose  $T_{n\times p}$ are $n\times p$ Toeplitz matrices with input sequence $\{ a_i\}_{ i \in \mathbb{Z} }$ obeying $(\ref{eqn:condition})$ and such that $a_{-j}=a_{j}$ for all $j \geq 1$ and $a_0 \equiv 0$. Let $X=T_{n\times p} (T_{n \times p})^t$. Assumption I
 Let $\{a_i\}_{i \in \N}$ be a sequence of random variables which satisfy Assumption I with $a_0 \equiv 0$. Suppose $T_{n \times p}$ is the $n \times p$ symmetric Toeplitz matrix with input sequence $\{\frac{a_i}{\sqrt{n}}\}_{i \in \N}$ and let $X=T_{n\times p} T'_{n \times p}$. Then for each $r \geq 1$ and $\lambda \in (0, \infty)$, $\E\left[\frac{1}{n} \Tr X^{r}\right]=M_{ r}+o(1/\sqrt{n})$ as $p/n \rightarrow \lambda$, $n \rightarrow \infty$, where
\begin{align}\label{eqn:2r-Moment_symTnp}
	M_{r}&=\sum_{\pi \in \mathcal{P}_2(2 r)} \int_{[0,1] \times [-m,m]^{|SP(\pi)|} \times[-\lambda,1]^{|DP(\pi)|}} \prod_{s=1}^{r} \chi_{[0,\lambda]}\left(x_0+ \sum_{i=1}^{2s-1} \epsilon_\pi(i) x_{\pi(i)}\right) \nonumber \\
	& \qquad \times \chi_{[0,1]}\left( x_0+ \sum_{i=1}^{2s} \epsilon_\pi(i) x_{\pi(i)}\right)  \prod_{l=0}^{r} \mathrm{~d} x_l,
\end{align}
with $m =\min \{\lambda,1\}$ and $SP(\pi), DP(\pi)$ are as in Definition \ref{def:SP+DP_partation}. 
\end{theorem}
\begin{proof}
First note from Lemma \ref{lem:trac_TT^t} that 
	\begin{equation}\label{eqn:E1/n tr X eq1}
			\E\left[\frac{1}{n} \Tr\left(X^{ r}\right)\right]=\frac{1}{n^{r+1}} \sum_{i=1}^n \sum_{j_{1}, \ldots, j_{2r}= -(p-1)}^{(n-1)}  \E\left[a_J\right] I(i,J,p,n) \delta_{0} (\sum_{q=1}^{2k}(-1)^{q} j_{q}),
	\end{equation}
where $a_J=\prod_{u=1}^{2r} a_{j_u}$ and $I(i,J,p,n)$ is as in (\ref{eqn:I(i,j,p,n)}).
%\begin{equation}\label{eqn:J_2r_0}
%\red{\mathcal{J}^0_{2r}= \{(j_{1}, \ldots, j_{2r}) \in \Z^{2r} : -(p-1) \leq j_q \leq (n-1) \mbox{ and } j_q \neq 0 \}.  }
%\end{equation} 

Observe that if for some vector $J=(j_1,j_2,\ldots, j_{2r})$, there exists a component $j_u$ such that $|j_u| \neq |j_v|$ for all $v \neq u$, then it follows from  Assumption I that $\E [a_J]=0$. Furthermore since $a_0 \equiv 0$, if $j_u=0$ for some $u$, then $\E[a_J] =0$. Therefore, (\ref{eqn:E1/n tr X eq1}) can be written as
\begin{equation}\label{eqn:E1/n tr X eq2}
\E\left[\frac{1}{n} \Tr\left(X^{ r}\right)\right]=\frac{1}{n^{r+1}} \sum_{i=1}^n \sum_{\pi \in \widehat{\mathcal{P}}(2r)}\sum_{J \in \Pi(\pi)}  \E\left[a_J\right] I(i,J,p,n) \delta_{0} (\sum_{q=1}^{2r}(-1)^{q} j_{q}),
\end{equation} 
where $\widehat{\mathcal{P}}(2r)$ is the set of all partitions of $[2r]$ such that each block has size greater than or equal to two and for a partition $\pi$, $\Pi(\pi)$ is the set of all vectors $J \in \{-(p-1),-(p-2), \ldots,-1,1,\ldots , (n-1)\}^{2r}$ such that $u \sim_{\pi} v$ if and only if $|j_u|=|j_v|$.

Note that for a fixed $\pi$, the number of choices for $J \in \Pi(\pi)$ is of the order $O(n^{|\pi|})$, where $|\pi|$ is the number of blocks of $\pi$. Consider $\pi \in \widehat{\mathcal{P}}(2r)$ such that there exist a block $B_i$ of size strictly greater than two, then it follows that the number of blocks of $\pi$ is less than or equal to $(r-1)$. Furthermore, since $\{a_i\}$ satisfies Assumption I, $\E\left[a_J\right] $ is constant for all $J \in \Pi(\pi)$ and is also bounded above by a constant independent of $J$. Thus the contribution of such terms to (\ref{eqn:E1/n tr X eq2}) is $O(\frac{1}{n^{r+1}} \times n \times n^{r-1})=o(\sqrt{n})$. This implies that for each $r$, a non-zero contribution in the limit arises due to pair-partitions of $[2r]$. Thus we have,
\begin{equation}\label{eqn:E1/n tr X^2r}
	\E\left[\frac{1}{n} \Tr\left(X^{r}\right)\right]=\frac{1}{n^{r+1}} \sum_{i=1}^n \sum_{\pi \in \mathcal{P}_2(2r)}\sum_{J \in \Pi(\pi)}  \E\left[a_J\right] I(i,J,p,n) \delta_{0} (\sum_{q=1}^{2r}(-1)^{q} j_{q}) + o\left(\frac{1}{\sqrt{n}}\right).
\end{equation} 
For $\pi \in \mathcal{P}_2(2r)$, consider the sets
\begin{align*}
	&\Pi_1(\pi)=\{ J \in \Pi(\pi): \text{ for all } u\sim_{\pi}v, j_u=\epsilon_{\pi}(v)j_v\} \text{ and }\\
	&\Pi_2(\pi)= \Pi(\pi) \setminus \Pi_1(\pi).
\end{align*}
For $\pi \in \mathcal{P}_2(2r)$ and fixed $1 \leq u \neq v \leq 2r$ such that $u \sim_{\pi} v$, consider the subset $\widetilde{\Pi}_2(\pi,u,v)=\{J \in \Pi_2(\pi): j_u \neq \epsilon_{\pi}j_v\}$. For $J \in \widetilde{\Pi}_2(\pi,u,v)$, note that the definition of $\epsilon_\pi$ implies that $(-1)^uj_u=(-)^vj_v$ and consequently the summation $\sum_{q=1}^{2r}(-1)^{q} j_{q}$ is zero if and only if $2(-1)^uj_u= \sum_{q=1 \atop q \neq u,v}^{2r}(-1)^{q} j_{q}$. Hence, once all other $j_q$ are chosen, the value of $j_u$ for which (\ref{eqn:E1/n tr X^2r}) is non-zero, is uniquely determined. Since $\cup_{u \sim_{\pi} v} \widetilde{\Pi}_2(\pi,u,v)=\Pi_2(\pi)$, it follows that the number of choices for $J \in \Pi_2(\pi)$, such that (\ref{eqn:E1/n tr X^2r}) is non-zero is at most $O(n^{r-1})$.

Furthermore, from the definition of $\epsilon_\pi$, it follows that for each $J \in \Pi_1(\pi)$, $\sum_{q=1}^{4r}(-1)^{q} j_{q}$ is always zero. Hence, we get that 
\begin{equation} \label{eqn:E1/n tr X eq3}
	\E\left[\frac{1}{n} \Tr\left(X^{r}\right)\right]=\frac{1}{n^{r+1}} \sum_{i=1}^n \sum_{\pi \in \mathcal{P}_2(2r)}\sum_{J \in \Pi_1(\pi)} I(i,J,p,n) +o\left(\frac{1}{\sqrt{n}}\right).
\end{equation} 
Suppose a block $B_j= \{u,v\}$ of $\pi$ is a same-parity block. Then for each $J \in \Pi_1(\pi)$, from the definition of $\Pi_1(\pi)$, it follows that $j_u=j_v$. And if $B_j$ is a different-parity block, then we get $j_u=-j_v$ and since $-(p-1) \leq j_u,j_v \leq (n-1)$, it follows that $-M \leq j_u \leq M$, where $M = \min \{p-1,n-1\}$. Without loss of generality, we use $j_1,j_2,\ldots,j_{|SP(\pi|)}$ to denote the elements of the same-parity blocks and  $j_{|SP(\pi)|+1},\ldots,j_{r}$ to denote the elements of the different-parity blocks. Substituting the expression for $I(i,J,p,n)$, 
%for each $\pi \in \mathcal{P}_2(2r)$, we have that the contribution to $\E[\frac{1}{n}\Tr X^{2r}]$ will be
\begin{align} \label{eqn: E Riemann}
	 \E[\frac{1}{n}\Tr( X^{r})]&= \hspace{-2mm} \sum_{\pi \in \PP_{2}(2r)}  \frac{1}{n^{r+1}} \sum_{i=1}^n \sum_{j_1,\ldots,j_{r}}\prod_{s=1}^{r} \chi_{[1,p]}(i+\sum_{\ell=1}^{2s-1} (-1)^\ell j_\ell )\,\, \chi_{[1,n]}(i+\sum_{\ell=1}^{2s} (-1)^\ell j_\ell )+o\left(\frac{1}{\sqrt{n}}\right)\nonumber\\
	&= \hspace{-2mm} \sum_{\pi \in \PP_{2}(2r)} \hspace{-1.5mm}\frac{1}{n^{r+1}} \sum_{i=1}^n \sum_{j_1,\ldots,j_{r}}\prod_{s=1}^{r} \chi_{[\frac{1}{n}\frac{p}{n}]}(i+\sum_{\ell=1}^{2s-1} (-1)^\ell j_\ell )\,\, \chi_{[\frac{1}{n},1]}(i+\sum_{\ell=1}^{2s} (-1)^\ell j_\ell )+o\left(\frac{1}{\sqrt{n}}\right),
\end{align} 
where $j_1,j_2,\ldots,j_{|SP(\pi)|}$ varies between $-(p-1)$ and $(n-1)$, and $j_{|SP(\pi)|+1},\ldots,j_{r}$ varies between $-M$ and $M$. Now, note that as $p,n \rightarrow \infty$ and $p/n \rightarrow \lambda$, the leading term in (\ref{eqn: E Riemann}) corresponding to $\pi \in \PP_{2}(2r)$ is the Riemann sum of the following integral
\begin{align*}
	\int_{[0,1] \times [-m,m]^{|SP(\pi)|} \times[-\lambda,1]^{|DP(\pi)|}} \prod_{s=1}^{r} \chi_{[0,\lambda]}\left(x_0+ \sum_{i=1}^{2s-1} \epsilon_\pi(i) x_{\pi(i)}\right)\chi_{[0,1]}\left( x_0+ \sum_{i=1}^{2s} \epsilon_\pi(i) x_{\pi(i)}\right)  \prod_{l=0}^{r} \mathrm{~d} x_l,
\end{align*} 
where $m =\min \{\lambda,1\}$ and $SP(\pi), DP(\pi)$ are as in Definition \ref{def:SP+DP_partation}. Therefore $\E[\frac{1}{n}\Tr( X^{r})]$ converges to $M_{r}$ for all $r \geq 1$ and this completes the proof.
\end{proof}
\begin{remark}\label{rem: Mr >0}
Note that the function under the integral in (\ref{eqn:2r-Moment_symTnp}) is a non-negative function. Furthermore, since $|\epsilon_\pi(i)|=1$ for all $i$, the function is strictly positive on the domain $\left[\frac{m}{4},\frac{3m}{4}\right] \times \left[-\frac{m}{8r},\frac{m}{8r}\right]^{r}$, where $m=\min \{\lambda,1\}$. It follows from here that $M_r>0$ for all $r \geq 1$. 
\end{remark}
%\vskip2pt
\subsection{\large \bf Fluctuation of $T_{n \times p}$  with diagonal entries as zero}  \label{sec:proof_thm_TT'}

In this section, we prove Theorem \ref{thm:LES_TT'}.
We first define some sets and relations which will appear in the proof of Theorem \ref{thm:LES_TT'}.
For a vector $J= (j_1,j_2, \ldots, j_p) \in \Z^p$, we define the multi-set $S_{J}$  as 
\begin{equation}\label{def:S_|J|}
	%S_J=\{ j_1, j_2, \ldots , j_p \}. % \mbox{ and } 
	S_{J}=\{ |j_1|, |j_2|, \ldots , |j_p| \}.
\end{equation}
%$J_r=(j_1^r, j_2^r,\ldots , j_p^r)$
For a sequence of vectors $J_1, J_2, \ldots $, we shall use the notation $J_r=(j_1^r, j_2^r,\ldots , j_p^r)$ to denote the components of $J_r$ and $S_{J_r}$ to denote the multi-set associated with $J_r$.

Now, we derive the limiting covariance structure of $\{w_k;k\geq 1\}$.
\begin{theorem} \label{thm:covarTT'}
Let $\{a_i\}_{i \in \N}$ be a sequence of random variables which satisfy Assumption I with $a_0 \equiv 0$. Suppose $T_{n \times p}$ is the $n \times p$ symmetric Toeplitz matrix with input sequence $\{\frac{a_i}{\sqrt{n}}\}_{i \geq 0}$ and let $w_{k_1},w_{k_2}$ be as defined in (\ref{eqn:wp_Hn_odd}). Then for positive integers $k_1,k_2$, 
	\begin{align}\label{eqn:sigma_p,q_TT'-sym}
		\lim_{n \rightarrow \infty \atop p/n \rightarrow \lambda>0} \Cov(w_{k_1},w_{k_2})= \sum_{\pi \in \PP_2(2k_1,2k_2)}  (f^{-}_I(\pi)+f^{+}_I(\pi)) + (\kappa-1) \sum_{\pi \in \PP_{2,4}(2k_1,2k_2)}  (f^-_{II}(\pi)+f^+_{II}(\pi)),
	\end{align}
where $f^-_I(\pi),f^+_I(\pi), f^-_{II}(\pi)$ and $f^+_{II}(\pi)$ are as given in (\ref{eqn:f_Ipi-_T}), (\ref{eqn:f_Ipi+_T}),  (\ref{eqn:f_II-pi_T}) and (\ref{eqn:f_II+pi_T}), respectively.
\end{theorem}

\begin{proof} 
First note from Lemma \ref{lem:trac_TT^t} that 
\begin{align} \label{eqn: E w_k_1 first}
&w_{k_1} =\frac{1}{n^{k_1+\frac{1}{2}}} \sum_{i=1}^{n} \sum_{j_{1}, \ldots, j_{2k_1}=-(p-1)}^{n-1} \Big( \prod_{r=1}^{2k_1} a_{j_{r}}-\E\big(\prod_{r=1}^{2k_1} a_{j_{r}}\big) \Big)   I(i,J_{1},p,n)  \delta_{0} (\sum_{q=1}^{2k_1}(-1)^{q} j_{q}), 
%&\E(w_{k_2}) =\frac{1}{n^{k_2+\frac{1}{2}}} \sum_{i=1}^{n} \sum_{j_{1}, \ldots, j_{2k_2}=-(p-1)}^{n-1} \prod_{r=1}^{2k_2} a_{j_{r}}   I(i,J_{k_2},p,n)  \delta_{0} (\sum_{q=1}^{2k_2}(-1)^{q} j_{q}) \label{eqn: E w_k_2 first}
\end{align} 
where $I(i,J_{1},p,n)$ is as defined in (\ref{eqn:I(i,j,p,n)}). It follows that 
\begin{align} \label{eqn: Ew_k_1 w_k_2 first}
	& \Cov(w_{k_1},w_{k_2}) \nonumber \\
	&=\frac{1}{n^{k_1+k_2+1}}\sum_{i=1 \atop i^\prime=1}^{n} \sum_{j_{1}, \ldots, j_{2k_1}=-(p-1) \atop j_{1}^\prime, \ldots, j_{2k_2}^\prime=-(p-1)}^{n-1} \Big[\E\big(\prod_{r=1}^{2k_1} a_{j_{r}}\prod_{r=1}^{2k_2} a_{j_{r}^\prime}\big)- \E\big(\prod_{r=1}^{2k_1} a_{j_{r}} \big) \E\big(\prod_{r=1}^{2k_2} a_{j_{r}^\prime}\big) \Big] g(i,i^\prime,J_1,J_2,p,n),
\end{align}
where  $J_1$ and $J_2$ are given by $J_1=(j_{1},j_2, \ldots, j_{2k_1})$, $J_2=(j_{1}^\prime,j_2^\prime, \ldots, j_{2k_2}^\prime)$, and
%; the functions $I(i,J_{1},p,n)$ and $I(i,J_{2},p,n)$ are as given in (\ref{eqn:I(i,j,p,n)}) and
\begin{equation}\label{eqn:g(J,Jprime,p,n)}
	g(i,i^\prime,J_1,J_2,p,n)= I(i,J_{1},p,n)I(i,J_{2},p,n) \delta_{0} (\sum_{q=1}^{2k_1}(-1)^{q} j_{q}) \delta_{0} (\sum_{q=1}^{2k_2}(-1)^{q} j_{q}^\prime).
\end{equation}

For a vector $J= (j_1,j_2,\ldots, j_{2k_1+2k_2}) \in \{-(p-1),\ldots,(n-1)\}^{2k_1+2k_2}$, we define $j_r^\prime=j_{2k_1+r}$ for all $1 \leq r \leq 2k_2$. In this proof, we maintain the convention that for a vector $J=(j_1,j_2,\ldots, j_{2k_1+2k_2}) \in \{-(p-1),\ldots,(n-1)\}^{2k_1+2k_2}$, $J_1=(j_1,j_2,\ldots, j_{2k_1})$ and $J_2=(j_{2k_1+1},j_{2k_1+2}\ldots, j_{2k_1+2k_2})$.  Furthermore, note that for a vector $J$, if the element $|j_u|$ appears only once in $S_J$ or if $j_u=0$, then  the summand corresponding to $J$ in (\ref{eqn: Ew_k_1 w_k_2 first}) is zero. Hence, we get
%We choose the following notation: For $1 \leq r \leq 2k_2$, we write $j_{2k_1+r}=j_r^\prime$. Consider a fixed choice of $J=(j_1,j_2,\ldots,j_{2(k_1+k_2)})$. Consider the equivalence relation defined on the set $[2(k_1+k_2)]$ given by, $r \sim_J s$ if $j_r=j_s$. Let $B_1,B_2,\ldots, B_u$ be the equivalence classes with respect to the relation $\sim_J$. Suppose $\hat{J}=(\hat{j}_1,\hat{j}_2,\ldots,\hat{j}_{2(k_1+k_2)})$ is another choice of the $(2k_1+2k_2)$ tuples such that relation $\sim_{\hat{J}}$ has equivalence classes $B_1,B_2,\ldots, B_u$. Then it follows from the independence of $(a_j)_{j \in \N}$ that the values of $\E(w_{k_1}),\E(w_{k_2})$ and $\E(w_{k_1})\E(w_{k_2})$ in (\ref{eqn: E w_k_1 first})-(\ref{eqn: E w_k_2 first}) are the same. Furthermore, if the size of one of the equivalence classes is one, then since $\E(a_j)=0$ for all $j$, it follows that from (\ref{eqn: E w_k_1 first})-(\ref{eqn: Ew_k_1 w_k_2 first}) that  $\E(w_{k_1}),\E(w_{k_2})$ and $\E(w_{k_1}w_{k_2})$ are all zero. Hence we get,
\begin{align}\label{eqn: Cov_TT^t one}
	& \Cov(w_{k_1},w_{k_2}) \nonumber\\
	&= \frac{1}{n^{k_1+k_2+1}}\sum_{i=1 \atop i^\prime=1}^{n} \sum_{\pi \in \widehat{\mathcal{P}}\left(2k_1+2k_2\right)} \sum_{J \in \Pi(\pi)} \Big[\E\Big(\prod_{r=1}^{2k_1+2k_2} a_{j_{r}}\Big)-\E\Big(\prod_{r=1}^{2k_1} a_{j_{r}}\Big)\E\Big(\prod_{r=2k_2+1}^{2k_2} a_{j_{r}^\prime}\Big) \Big] g(i,i^\prime,J_1,J_2,p,n) \nonumber \\
	&= \frac{1}{n^{k_1+k_2+1}}  \sum_{\pi \in \widehat{\mathcal{P}}\left(2k_1+2k_2\right)} \sum_{i=1 \atop i^\prime=1}^{n} |\Pi(\pi)| V(\pi), \,\text{say} ,
\end{align}
where $\widehat{\mathcal{P}}\left(2k_1+2k_2\right)$ is the set of all partitions on $[2k_1+2k_2]$ where each block has size greater than or equal to 2, and  $\Pi(\pi)$ is the set of all vectors $J \in \{-(p-1),-(p-2), \ldots,-1,1,\ldots, (n-1)\}^{2k_1+2k_2}$ such that $u \sim_{\pi} v$ if and only if $|j_u|=|j_v|$. Note that the last summation follows from the fact that the summand in (\ref{eqn: Ew_k_1 w_k_2 first}) is equal for all $J \in \Pi(\pi)$. Here $|\Pi(\pi)|$ denotes the cardinality of the set $\Pi(\pi)$.

Now, we find the set of all partitions $\pi$ such that the asymptotic contribution of $\pi$ to $\Cov(w_{k_1}w_{k_2})$ is non-zero. Note that once the partition is fixed, the degree of freedom for choosing $j_1,j_2,\ldots,j_{2(k_1+k_2)}$ is equal to $|\pi|$, the number of 
blocks of the partition  $\pi$. Now, we consider different cases based on the number of blocks in the partition $\pi$.
\vskip3pt
\noindent \textbf{Case I:} The number of blocks of $\pi$ is strictly less than $k_1+k_2-1$.

\noindent In this case, we have that $|\Pi(\pi)|$ is of the order $O(n^{|\pi|})< O(n^{k_1+k_2-1})$. From Assumption I, we have that for each fixed $r$, there exists $M>0$ such that $\E |a_j|^r \leq M$ for all $j \in \N$. Therefore, the summand in (\ref{eqn: Cov_TT^t one}) is bounded above by a constant. Hence, we get that for each $\pi$ following the condition of Case I:
\begin{equation*}
\frac{1}{n^{k_1+k_2+1}}  \sum_{J \in \Pi(\pi)} \sum_{i=1 \atop i^\prime=1}^{n} | \Pi(\pi)| V(\pi)=O\left(\frac{1}{n^{k_1+k_2+1}} n^2 n^{| \Pi(\pi)|}\right)=o(1).
\end{equation*}
\vskip4pt
\noindent \textbf{Case II:} The number of blocks of $\pi$ is $k_1+k_2-1$.

\noindent Consider $\pi \in \widehat{\mathcal{P}}(2k_1+2k_2)$ such that the number of blocks of $\pi$ is $k_1+k_2-1$. Suppose $\pi$ contains a block $B_j$ such that $|(B_j \cap [2k_1])|=1$ or $|\left(B_j \cap \left([2(k_1+k_2)]\setminus [2k_1]\right)\right)|=1$. Without loss of generality, suppose $|(B_j \cap [2k_1])|=1$ and let $\{b\}=B_j \cap [2k_1]$. Observe that if all $j_q$, except $j_b$ are chosen, then the number of choices of $j_b$ for which the summand in (\ref{eqn: Ew_k_1 w_k_2 first}) is non-zero is at most one. So, in this case, there is a loss of one degree of freedom for choosing $j_q$'s and thus $| \Pi(\pi)|=O(n^{|\pi|-1})$ for all $\pi$ such that $V(\pi) \neq 0$, and therefore
\begin{equation*}
\frac{1}{n^{k_1+k_2+1}}  \sum_{J \in \Pi(\pi)} \sum_{i=1 \atop i^\prime=1}^{n} | \Pi(\pi)| V(\pi)=o(1).
\end{equation*}
As a result, in this case for the contribution of $\pi$ to be non-zero in the limit, each cross-matched block must have at least four elements. Now, observe that the number of blocks of $\pi$ is $k_1+k_2-1$ and there is at least one cross-matching in $\pi$. Thus it follows that $\pi$ has a non-zero contribution in limit only if all the following conditions are satisfied:
\begin{enumerate}[(i)]
	\item $\pi$ has only one-cross matched block, say $B_j$.
	\item $|B_j|=4$ and, two elements of $B_j$ belong to $[2k_1]$ and the remaining two elements of $B_j$ belong to $[2(k_1+k_2)\setminus 2k_1]$.
	\item All other block of $\pi$ have cardinality 2.
\end{enumerate}
In other words, $\pi \in \widehat{\mathcal{P}}(2k_1+2k_2)$ makes a non-zero contribution in the limit of (\ref{eqn: Cov_TT^t one}) only if $\pi \in \mathcal{P}_{2,4}(2k_1,2k_2)$. 

Consider a fixed $\pi \in {\mathcal{P}}_{2,4}(2k_1+2k_2)$. From Assumption I, we have that $\E a_j^2=1$ and $\E a_j^4=\kappa$ for all $j$. Thus from (\ref{eqn: Cov_TT^t one}), we have 
\begin{align}\label{eqn: Cov eq2}
	\frac{1}{n^{k_1+k_2+1}}  \sum_{J \in \Pi(\pi)} \sum_{i=1 \atop i^\prime=1}^{n} |\Pi(\pi)| V(\pi) &=(\kappa-1)	\frac{1}{n^{k_1+k_2+1}}  \sum_{i=1 \atop i^\prime=1}^{n}\sum_{J \in \Pi(\pi)} g(i,i^\prime,J_1,J_2,p,n),
	%&=(\kappa-1) \frac{1}{n^{k_1+k_2+1}}  \sum_{i=1 \atop i^\prime=1}^{n}\sum_{x_1,x_2,\ldots,x_{k_1+k_2-1}=1}^n \sum_{J} g(J,p,n)\label{eqn: Cov eq2},
\end{align}
where $\Pi(\pi)$ is as defined in (\ref{eqn: Cov_TT^t one}).

%Consider a fixed choice of $x_1,x_2,\ldots,x_{k_1+k_2-1}$ and $\pi \in \mathcal{P}_{2,4}(2k_1,2k_2)$ with blocks $B_1,B_2,\ldots, B_{k_1+k_2-1}$ with $|B_i|=4$. Let $J=(J_1,J_2) \in \Pi(\pi)$ be a vector such that $|j_{r}|=x_{\pi(r)}$ for all $r$. By definition of $\mathcal{P}_{2,4}(2k_1,2k_2)$, we have that for each $j$, $|B_\ell \cap S_{J_u}| \in \{0,2\}$ for $u=1,2$ and if $B_\ell \cap S_{J_u}=\{u_1,u_2\}$, then $|j_{u_1}|=|j_{u_2}|=x_\ell$. By an argument similar to the proof in Theorem \ref{thm: Moment TT^t}, it follows that for each $J\in \Pi(\pi)$ non-zero contribution occurs only under the additional condition $j_{u_1}=(-1)^{u_1+u_2+1}j_{u_2}$.

Consider a fixed $\pi \in \mathcal{P}_2(2k_1,2k_2)$ with blocks $B_1,B_2,\ldots, B_{k_1+k_2-1}$ with $|B_i|=4$ and $J=(J_1,J_2) \in \Pi(\pi)$. By the definition of $\mathcal{P}_{2,4}(2k_1,2k_2)$, we have that for each $u$, $|(B_\ell \cap S_{J_u})| \in \{0,2\}$ for $u=1,2$ and let $B_\ell \cap S_{J_u}=\{u_1,u_2\}$. By an argument similar to the proof of Theorem \ref{thm: Moment TT^t}, it follows that a non-zero contribution occurs only under the additional condition $j_{u_1}=(-1)^{u_1+u_2+1}j_{u_2}$. This is possible only when
%To prove that, without loss of generality, suppose $B_\ell \cap S_{J_1}=\{u_1,u_2\}$ and $j_{u_1}=j_{u_2}$. Then it follows that the equation $\delta_{0} (\sum_{q=1}^{k_1}(-1)^{q} j_{q})$ is non-zero only for one value of $j_{u_1}$, and hence only one value of $x_\ell$. Thus the number of choices of $x_1,x_2,\ldots,x_{k_1+k_2-1}$ for which (\ref{eqn: Cov eq2}) is non-zero is of the order $O(n^{k_1+k_2-2})$, which implies that the contribution of all such terms converges to zero in limit.
%This implies that for each $\ell \neq i$, $u_1,u_2 \in B_\ell$, $\tau_\pi(u_1)j_{u_1}=\tau_\pi(u_2)j_{u_2}=\pm x_\ell$, where $\tau_\pi$ is as defined in (\ref{eqn:tau map}). Thus, if we choose $x_\ell \in \{\pm 1, \pm 2,\ldots, \pm n\}$, then 
\begin{align}\label{eqn:tau condition1}
	\tau_\pi(r)j_{r}=\tau_\pi(s)j_{s} \text{ for } B_{\ell}=\{r,s\},
\end{align}
and for $B_i=\{r_1,s_1,r_2,s_2\}$ with $r_1<s_1<r_2<s_2$,
\begin{align}
	&\tau_\pi(r_1)j_{r_1}=\tau_\pi(s_1)j_{s_1}=\tau_\pi(r_2)j_{r_2}=\tau_\pi(s_2)j_{s_2} \label{eqn:tau condition2}\\	\text{ or } &\tau_\pi(r_1)j_{r_1}=\tau_\pi(s_1)j_{s_1}=-\tau_\pi(r_2)j_{r_2}=-\tau_\pi(s_2)j_{s_2}. \label{eqn:tau condition3}
\end{align}
Note that a vector $(J_1,J_2)$ obeys (\ref{eqn:tau condition3}) if and only if $(J_1,-J_2)$ obeys (\ref{eqn:tau condition2}). Furthermore, note that a vector $(J_1,J_2)$ obeys (\ref{eqn:tau condition1}) if and only if $(J_1,-J_2)$ obeys (\ref{eqn:tau condition1}), and $(J_1,J_2) \in \Pi(\pi)$ if and only if $(J_1,-J_2) \in \Pi(\pi)$. On combining these ideas with (\ref{eqn: Cov eq2}), we get that the contribution of $\pi \in \mathcal{P}_{2,4}(2k_1,2k_2)$ to $\Cov(w_{k_1},w_{k_2})$ is 
\begin{equation}\label{eqn: Cov eq 3}	
	(\kappa-1) \frac{1}{n^{k_1+k_2+1}}  \sum_{i=1 \atop i^\prime=1}^{n} \left(\sum_{J \in \Pi_1(\pi)} g(i,i^\prime,J_1,J_2,p,n)+\sum_{J \in \Pi_2(\pi)} g(i,i^\prime,J_1,J_2,p,n)\right)+o(1),
\end{equation}
where 
\begin{align*}
	&\Pi_1(\pi)=\{J=(J_1,J_2) \in \Pi(\pi): (J_1,J_2) \text{ obeys } (\ref{eqn:tau condition1}) \text{ and } (\ref{eqn:tau condition2})\}, \\
	&\Pi_2(\pi)=\{J=(J_1,J_2) \in \Pi(\pi): (J_1,-J_2) \text{ obeys } (\ref{eqn:tau condition1})\text{ and } (\ref{eqn:tau condition2})\}.
\end{align*}

Note that by conditions (\ref{eqn:tau condition1}) and (\ref{eqn:tau condition2}), it follows that $\delta_{0}  (\sum_{q=1}^{2k_1}(-1)^{q} j_{q})$ and $ \delta_{0} (\sum_{q=1}^{2k_2}(-1)^{q} j_{q}^\prime)$ are equal to one for all $J \in \Pi_1(\pi)$ and $J \in \Pi_2(\pi)$. Therefore, the first term of (\ref{eqn: Cov eq 3}) is the Riemann sum of $f_{II}^+(\pi)$ and the second term of (\ref{eqn: Cov eq 3}) is the Riemann sum of $f_{II}^-(\pi)$. Hence, the contribution in this case is 
$$(\kappa-1) \sum_{\pi \in \PP_{2,4}(2k_1,2k_2)}  (f^-_{II}(\pi)+f^+_{II}(\pi)).$$
 \vskip3pt
\noindent\textbf{Case III:} The number of blocks of $\pi$ is $k_1+k_2$.

\noindent In this case, we have that each block is of size 2. Note that in this case, (\ref{eqn: Cov_TT^t one}) is equal to zero if all blocks of $\pi$ are either subsets of $[2k_1]$ or $[2(k_1+k_2)]\setminus [2k_1]$ . Hence, to have non-zero contribution in limit, $\pi$ should be an element of $\mathcal{P}_2(2k_1,2k_2)$. For $\pi \in \mathcal{P}_2(2k_1,2k_2)$, we have 
$$\E\left(\prod_{r=1}^{2k_1+2k_2} a_{j_{r}}\right)=1, \E\left(\prod_{r=1}^{2k_1} a_{j_{r}}\right)=0 \mbox{ and } \E\left(\prod_{r=2k_2+1}^{2k_2} a_{j_{r}^\prime}\right)=0.$$
 Thus, we get that for $\pi \in \mathcal{P}_2(2k_2,2k_2)$,
\begin{equation*}
	\frac{1}{n^{k_1+k_2+1}}  \sum_{J \in \Pi(\pi)} \sum_{i=1 \atop i^\prime=1}^{n} |\Pi(\pi)| V(\pi)=	\frac{1}{n^{k_1+k_2+1}}  \sum_{i=1 \atop i^\prime=1}^{n}\sum_{J \in \Pi(\pi)} g(i,i',J_1,J_2,p,n),
\end{equation*}
where $g(i,i',J_1,J_2,p,n)$ is as defined in (\ref{eqn: Ew_k_1 w_k_2 first}). Let $u$ be the smallest integer such that $u \sim_{\pi} v$ for some $ 1 \leq u \leq 2k_1 < v \leq 2(k_1+k_2)$. Then note that $g(i,i',J_1,J_2,p,n)$ is non-zero only if 
$$j_u=\sum_{i=1 \atop i \neq u}^{2k_1}(-1)^i j_i \mbox{ and } j_v=\sum_{i=2k_1+1 \atop i \neq u}^{2k_1+2k_2}(-1)^i j_i.$$ Thus $j_u$ and $j_v$ are determined by other $j_q$'s and $(-1)^uj_u+(-1)^v j_v$ might be non-zero. By an argument similar to Case II, it follows that a non-zero contribution occurs only when for $(J_1,J_2) \in \Pi(\pi)$, either $(J_1,J_2)$ or $(J_1,-J_2)$ obeys the equation
\begin{equation}\label{eqn: epsilon_cond (J_1,J_2)}
		\epsilon_\pi(u_1)j_{u_1}=\epsilon_\pi(u_2)j_{u_2}=x_\ell,
\end{equation}
where $j_i \in \{-(p-1),-(p-2),\ldots, -1, 1, 2, \ldots, n-1\}$ for each $i$ and $\epsilon_{\pi}$ is as defined in (\ref{eqn: epsilon map}). This implies that as $p,n \rightarrow \infty $ and $p/n \rightarrow \lambda$, the contribution when $(J_1,-J_2)$ obeys (\ref{eqn: epsilon_cond (J_1,J_2)}) is $f_I^-(\pi)$ and the contribution when $(J_1,J_2)$ obeys (\ref{eqn: epsilon_cond (J_1,J_2)}) is $f_I^+(\pi)$. This completes the proof of the Theorem \ref{thm:covarTT'}.
\end{proof}

Now we state some more notations and results  which will be used in the proof of Theorem \ref{thm:LES_TT'}.
% Theorem \ref{thm:toeprocess}. 
%For a vector $J = (j_1, j_2, \ldots, j_k) \in \Z^k$, define 
First recall the notion of multi-set $S_J$ for a given vector $J$ from (\ref{def:S_|J|}).
%\begin{equation}\label{def:S_j}
%	S_{J} = \{|j_1|, |j_2|, \ldots, |j_k|\}.
%\end{equation}
%Note that in $S_{J}$, elements are appearing with their multiplicities.
\begin{definition}\label{def:connected}
	Two vectors $J =(j_1, j_2, \ldots, j_k)\in \Z^k$ and $J' = (j'_1, j'_2, \ldots, j'_\ell)\in \Z^\ell$ are said to be \textit{connected} if
	$S_{J}\cap S_{J'} \neq \emptyset$.
\end{definition}
\begin{definition}\label{def:cluster}
	Given a set of vectors $S= \{J_1, J_2, \ldots, J_\ell \}$, where $J_i \in \Z^{k_i}$ for $1 \leq i \leq \ell$, a subset $C=\{J_{n_1}, J_{n_2}, \ldots, J_{n_k}\}$ of $S$ is called a \textit{cluster} if it satisfies the following two conditions: 
	\begin{enumerate}
		\item[(i)] For any pair $J_{n_i}, J_{n_j}$ from  $C$ one can find a chain of vectors from 
		$C$, which starts with $J_{n_i}$ and ends with $J_{n_j}$ such that any two neighbouring vectors in the chain are connected.
		\item[(ii)] The subset $C$ cannot  be enlarged to a subset which preserves condition (i).
	\end{enumerate}
\end{definition}
Now note that if the diagonal entries of Toeplitz matrices are zero, then the trace formula of Lemma \ref{lem:trac_TT^t}  can be written as
\begin{align} \label{def:trace_Tn_A}
	\Tr( X^{(k)} \cdots X^{(1)})
	= \sum_{i=1}^{n} \sum_{A_k}  \prod_{r=1}^{k} a^{(r)}_{j_{2r-1}} a^{(r)}_{j_{2r}}  I(i,J,p,n),
\end{align}
where $I(i,J,p,n)$ is as in (\ref{eqn:I(i,j,p,n)}), $X^{(r)}=T^{(r)}_{n\times p} (T^{(r)}_{n \times p})'$ and 
\begin{align}\label{eqn:A_k}
	A_k= \big\{(j_1,\ldots, j_{2k}) \in \{-(p-1),\ldots, -1, 1,\ldots, (n-1)\}^{2k} : \sum_{q=1}^{2k}(-1)^{q} j_{q}=0 \big\}.
\end{align}

The following lemma guarantee that the cardinality of a clusters with length greater than two is negligible when $n$ becomes large.
\begin{lemma}\label{lem:B_{P_l}} 
 Suppose $ B_{K_\ell}$ be the subset of all $ (J_1,J_2,\ldots,J_\ell)\in A_{k_1} \times A_{k_2} \times \cdots \times A_{k_\ell}$ such that  
	\begin{enumerate} 
		\item[(i)] $\{J_1, J_2, \ldots, J_\ell\} $ forms a cluster, 
		\item[(ii)] each element in  $\displaystyle{\cup_{i=1}^{\ell} S_{J_i} }$ has  multiplicity greater than or equal to two.
		% 	\item[(iii)] $0 \notin \displaystyle{\cup_{i=1}^{\ell} S_{J_i} }$.
	\end{enumerate} 
	Then for $\ell \geq 3$,
	\begin{equation*}
		|B_{K_\ell }| = o \big({(n+p)}^{k_1+k_2 + \cdots + k_\ell - \frac{\ell}{2} }\big)
	\end{equation*}
and $|B_{K_2 }| = O \big({(n+p)}^{k_1+k_2-1}\big)$.
\end{lemma}
For the proof of Lemma \ref{lem:B_{P_l}}, we refer the reader to Lemma 5.3 of \cite{liu2012fluctuations}. The idea of the proof is same and so we skip the proof here.
The following lemma is an easy consequence of Lemma \ref{lem:B_{P_l}} and is one of the main ingredients for the proof of Theorem \ref{thm:LES_TT'}.

\begin{lemma}\label{lem:maincluster_TT'}
	Suppose  $\{x_i\}_{i \geq 1}$ is a sequence of random variables which satisfy Assumption I. Then for $n,p \rightarrow \infty$ with $p/n \rightarrow \lambda >0$ and $\ell \geq 3,$
	\begin{equation}\label{equation:maincluster_rev}
		\frac{1}{ n^{k_1+k_2+ \cdots + k_\ell - \frac{\ell}{2}}} \sum_{J_1 \in A_{k_1}, \ldots, J_\ell \in A_{k_\ell}, \atop  \{J_1, \ldots, J_\ell \} \text{ forms a cluster}} \E\Big[\prod_{r=1}^{\ell}\Big(x_{J_r} - \E(x_{J_r})\Big)\Big] = o(1),
	\end{equation}
	where
	$J_r=(j^{r}_{1}, j^{r}_{2},\ldots, j^{r}_{2k_r})\in A_{k_r}, \ A_{k_r} \mbox{ as in (\ref{eqn:A_k}) and} \ x_{J_{r}} = \prod_{s=1}^{2k_r} x_{j_{s}^r}.$	
\end{lemma}
\begin{proof} First note that $\E(x_i)=0$ for each $i$ therefore  $\E\Big[\prod_{r=1}^{\ell}\Big(x_{J_r} - \E(x_{J_r})\Big)\Big]$ will be non-zero only if each $x_i$ appears at least twice in the collection $\{ x_{j^{r}_{1}}, x_{j^{r}_{2}}, \ldots ,x_{j^{r}_{2k_r}} ; 1\leq r\leq \ell\}$. Thus
	\begin{equation}\label{eqn:equality_reduction_rev}
		\sum_{J_1 \in A_{k_1}, \ldots, J_\ell\in A_{k_\ell}, \atop  \{J_1, \ldots, J_\ell \} \text{ forms a cluster}} \E\Big[\prod_{r=1}^{\ell}\Big(x_{J_r} - \E(x_{J_r})\Big)\Big] =\sum_{(J_1, J_2,\ldots, J_\ell)\in B_{K_\ell}} \hspace{-3pt} \E\Big[\prod_{r=1}^{\ell}\Big(X_{J_r} - \E(X_{J_r})\Big)\Big],
	\end{equation} 
	where $B_{K_\ell}$ is as in Lemma \ref{lem:B_{P_l}}. By Assumption I, we have  that moments of $x_i$'s are bounded. 
%	\begin{equation*}\label{eqn:higher moment finite for reverse}
%		\E(x^2_i)=1 \mbox{ and } \sup_{i \geq 1}\E(|x_i|^{k})= \alpha_k < \infty \mbox{ for } k \geq 3.
%	\end{equation*}  
	Thus for all $(J_1, J_2, \ldots, J_\ell)\in A_{k_1}\times A_{k_2}\times \cdots \times A_{k_\ell}$, we have 
	\begin{equation}\label{eqn:modulus finite_rev}
		\Big|\E\big[\prod_{r=1}^{\ell}\big(x_{J_r} - \E(x_{J_r})\big)\big]\Big|\leq \beta_\ell,
	\end{equation} 
 where $\beta_\ell>0$ depends only on $k_1,k_2,\ldots,k_\ell$. Now
note from \eqref{eqn:equality_reduction_rev} and \eqref{eqn:modulus finite_rev} that
	\begin{align*}
		\sum_{J_1 \in A_{k_1}, \ldots, J_{\ell} \in A_{k_\ell}, \atop  \{J_1, \ldots, J_\ell \} \text{ forms a cluster}} \Big|\E\big[\prod_{r=1}^{\ell}\big(x_{J_r} - \E(x_{J_r})\big)\big]\Big|
		& \leq  \sum_{(J_1,J_2,\ldots,J_\ell)\in B_{K_\ell}} \beta_{\ell} 
		\ = | B_{K_\ell}| \times \beta_\ell.
	\end{align*}
Thus using Lemma \ref{lem:B_{P_l}} in the above expression, we get   
	\begin{align*}
	\sum_{J_{1} \in A_{k_1}, \ldots, J_{\ell} \in A_{k_\ell}, \atop  \{J_1, \ldots, J_\ell \} \text{ forms a cluster}} \Big|\E\big[\prod_{r=1}^{\ell}\big(x_{J_r} - \E(x_{J_r})\big)\big]\Big| = o \big((n+p)^{k_1+k_2 + \cdots + k_\ell -\frac{\ell}{2} }\big).
	\end{align*}
	Since $p/n \tends \lambda \in (0,\infty)$, $p=O(n)$ and the above expression gives (\ref{equation:maincluster_rev}). This completes the proof of the lemma. 
\end{proof}

Now we are ready to prove Theorem \ref{thm:LES_TT'}  using the above lemmata and Theorem \ref{thm:covarTT'}. 

\begin{proof}[Proof of Theorem \ref{thm:LES_TT'}] We use method of moments and  Wick's formula to prove Theorem \ref{thm:LES_TT'}. First recall from the method of moments that, to prove $w_Q \stackrel{d}{\rightarrow} N(0,\sigma_{Q}^2)$, it is sufficient to show that 
	\begin{align} \label{eqn:moment w_Q for reverse}
	\lim_{n \rightarrow \infty \atop p/n \rightarrow \lambda>0} \E[ (w_Q)^\ell] = \E[ (N(0, \sigma^2_{Q}))^\ell ] \ \ \forall \ \ell=1,2, \ldots.
	\end{align}
	So, to prove (\ref{eqn:moment w_Q for reverse}), it is enough to show that, for $k_1, k_2, \ldots , k_\ell \geq 1$,
	\begin{equation}\label{eqn:limE[w1+k]}
	\lim_{n \rightarrow \infty \atop p/n \rightarrow \lambda>0}  \E[w_{k_1}w_{k_2} \cdots w_{k_\ell}]=\E[N_{k_1}N_{k_2} \cdots N_{k_\ell}],	
	\end{equation}
	where $\{N_{k}\}_{k \geq 1}$ is a centred Gaussian family with covariance $\sigma_{k_1,k_2}$ as in (\ref{eqn:sigma_p,q_TT'-sym}).
	%that is, $\E[N_{p},N_{q}]= \sigma_{p,p}$, where  

	Note from the trace formula (\ref{def:trace_Tn_A}) that
	\begin{align*}
		w_{k_r}  &= \frac{1}{\sqrt{n}} \Big(\Tr(T_{n\times p} T'_{n \times p})^{k_r} - \E[\Tr(T_{n\times p} T'_{n \times p})^{k_r}]\Big) \\
		&= \frac{1}{n^{k_r +\frac{1}{2}}} \sum_{i_r=1}^{n} \sum_{A_{k_r}} \big( a_{J_{k_r}} - \E[a_{J_{r}}]\big)  I(i_r,J_{r},p,n),
	\end{align*}
where $J_{r}=(j^{r}_{1},j^{r}_{2},\ldots, j^{r}_{2k_r})\in A_{k_r}$, $A_{k_r}$ as in (\ref{eqn:A_k}), $a_{J_{r}} = \prod_{s=1}^{2k_r} a_{j_{s}^r}$ and $ I(i_r,J_{r},p,n)$ is as in (\ref{eqn:I(i,j,p,n)}).
	Therefore 
	\begin{align}\label{eqn:expectation_thm2_rev}
		&\quad \E[w_{k_1}w_{k_2} \cdots w_{k_\ell}] \nonumber \\
		& \quad = \frac{1}{n^{k_1 + k_2 + \cdots +k_\ell + \frac{\ell}{2}}} \sum_{i_1, \ldots, i_\ell=1}^{n} \sum_{A_{k_1}, \ldots, A_{k_\ell}} \E\Big[  \prod_{r=1}^{\ell} (a_{J_{r}} - \E[a_{J_{r}}]) \Big] \prod_{r=1}^{\ell} I(i_r,J_{r},p,n).
	\end{align}

	First note that for fixed $J_{1},J_{2},\ldots,J_{\ell}$, if there exists a $s\in\{1,2,\ldots,\ell\}$ such that $J_{s}$ is not connected with any $J_{r}$ for $r\neq s$, then due to the independence of the entries $\{a_i\}_{i\geq 1}$, we have
	$$\E\Big[  \prod_{r=1}^{\ell} (a_{J_{r}} - \E[a_{J_{r}}]) \Big]=0.$$
	Thus for a non-zero contribution in the limit, each cluster in $\{J_{1},J_{2},\ldots,J_{\ell}\}$ must have length greater than or equal to two, where the notion of cluster is defined in Definition \ref{def:cluster}. Suppose the vectors $J_{1},J_{2},\ldots,J_{\ell}$ decomposes into clusters $T_1,T_2,\ldots,T_s$ with  $|T_i|\geq 2$ for all $1\leq i \leq s$, where $|T_i|$ denotes the length of the cluster $T_i$. Observe that $\sum_{i=1}^s | T_i|=\ell$.   
	
	If there exists  a cluster $T_j$ among $T_1,\ldots,T_s$ such that $|T_j |\geq 3$, then from Theorem \ref{thm:covarTT'} and Lemma \ref{lem:maincluster_TT'}, we have  
	\begin{align*} % \label{eqn:E_odd_w_p} 
		 \frac{1}{n^{k_1 + k_2 + \cdots +k_\ell + \frac{\ell}{2}}} \sum_{i_1, \ldots, i_\ell=1}^{n} \sum_{A_{k_1}, \ldots, A_{k_\ell}} \E\Big[  \prod_{r=1}^{\ell} (a_{J_{r}} - \E[a_{J_{r}}]) \Big] =o(1).
	\end{align*}  
	Since for each $r$, $| I(i_r,J_{r},p,n)| \leq 1$, we have
		\begin{align*} %\label{eqn:E_odd_w_p} 
		\frac{1}{n^{k_1 + k_2 + \cdots +k_\ell + \frac{\ell}{2}}} \sum_{i_1, \ldots, i_\ell=1}^{n} \sum_{A_{k_1}, \ldots, A_{k_\ell}} \Big| \E\Big[  \prod_{r=1}^{\ell} (a_{J_{r}} - \E[a_{J_{r}}]) \Big] \prod_{r=1}^{\ell} I(i_r,J_{r},p,n) \Big| =o(1).
	\end{align*} 
	 Thus, if $\ell$ is odd then there will be a cluster of odd length and hence 
	$$	\lim_{n \rightarrow \infty \atop p/n \rightarrow \lambda>0}  \E[w_{k_1}w_{k_2} \cdots w_{k_\ell}] = 0.$$ 
	
	Now suppose $\ell$ is even. From arguments similar to those of the $\ell$ odd case, we have that the contribution from $J_{1},J_{2},\ldots,J_{\ell}$ to  $ \E[w_{k_1} w_{k_2}\cdots w_{k_\ell}]$ is $O(1)$ only when $J_{1},J_{2},\ldots,J_{\ell}$  decomposes into clusters of length 2. Therefore from \eqref{eqn:expectation_thm2_rev}, we get
	\begin{align} \label{eqn:multisplit}
		& \E[w_{k_1}w_{k_2} \cdots w_{k_\ell}] \nonumber\\
		%& =\lim_{n\to\infty} \frac{1}{n^{p_1 + p_2 + \cdots +p_\ell -\frac{\ell}{2}}} \sum_{A_{2p_1}, A_{2p_2}, \ldots, A_{2p_\ell}} \E\big[ (X_{J_1} - \E X_{J_1}) (X_{J_2} - \E X_{J_2})  \cdots (X_{J_\ell} - \E X_{J_\ell})\big] \nonumber \\
		& =	\frac{1}{n^{k_1 + \cdots +k_\ell + \frac{\ell}{2}}} \sum_{\pi \in \mathcal P_2(\ell)} \prod_{r=1}^{\ell/2}  \sum_{i_{y(r)}, i_{z(r)}=1}^{n} \sum_{A_{k_{y(r)}}, A_{k_{z(r)}}} \E\big[ (a_{J_{y(r)}} - \E [a_{J_{y(r)}}]) (a_{J_{z(r)}} - \E [a_{J_{z(r)}]})\big] \nonumber \\
		& \qquad \qquad \times I(i_{y(r)},J_{y(r)},p,n) I(i_{z(r)},J_{z(r)},p,n) + o(1) \nonumber \\
		& = \sum_{\pi \in P_2(\ell)} \prod_{i=1}^{\ell/2}  \E[w_{k_{y(i)}} w_{k_{z(i)}}] +o(1),
	\end{align}
	where  $\pi = \big\{ \{y(1), z(1) \}, \ldots , \{y(\frac{\ell}{2}), z(\frac{\ell}{2})  \} \big\}\in \mathcal P_2(\ell)$ and $\mathcal P_2(\ell)$ is the set of all pair partition of $[\ell]$.
	Now from Theorem \ref{thm:covarTT'} and (\ref{eqn:multisplit}), we get
	\begin{align*}
		\lim_{n \rightarrow \infty \atop p/n \rightarrow \lambda>0}  \E[w_{k_1}w_{k_2} \cdots w_{k_\ell}]
		 =\sum_{\pi \in P_2(\ell)} \prod_{i=1}^{\ell/2} \lim_{n \rightarrow \infty \atop p/n \rightarrow \lambda>0} \E[w_{k_{y(i)}} w_{k_{z(i)}}]
		& =\sum_{\pi \in \mathcal P_2(\ell)} \prod_{i=1}^{\ell/2} \E[N_{k_{y(i)}} N_{k_{z(i)}} ] \\
		&=\E[ N_{k_1}N_{k_2} \cdots N_{k_\ell} ],
	\end{align*}
	where the last equality arises due to  Wick's formula. This completes the proof of Theorem \ref{thm:LES_TT'}.
\end{proof}

\subsection{\large \bf Fluctuation of $T_{n \times p}$  with non-zero diagonal entries}  \label{sec:proof_thm_TT'_nonzero}

	In  Theorem \ref{thm:LES_TT'}, we considered $a_0\equiv 0$. In this section, we consider the case where $a_0$ a non-zero random variable. First, we define the following notations:
	For $k\geq 1$, define 
	\begin{align*} %\label{eqn:w_p_tilde}
		\tilde{T}_{n \times p} & := T_{n \times p} + \frac{a_0}{\sqrt{n}} I_{n \times p}, \nonumber \\
		\tilde{w}_{k} & := \frac{1}{\sqrt{n}} \big[ \Tr(\tilde{T}_{n \times p}\tilde{T}'_{n \times p})^{k} - \E[\Tr(\tilde{T}_{n \times p} \tilde{T}'_{n \times p})^{k}] \big],
	\end{align*}
	where $T_{n \times p}$ is the symmetric Toeplitz matrix whose entries are $\{\frac{a_i}{\sqrt{n}}\}_{i \geq 0}$ with zero diagonal entries ($a_0=0$) and $I_{n \times p}$ is an $n \times p$ matrix given by $(\delta_{ij})_{i,j=1}^{n,p}$.
	
\begin{theorem}\label{thm:LES_TT'_a0}
 Let $\{a_i\}_{i \geq 0}$ be a sequence of random variables which satisfy Assumption I. Suppose $\tilde{T}_{n \times p}$ is defined as above.
Then for every  $k\geq 1$, as $ n,p \tends \infty$ with $p/n \tends \lambda>0$,
	\begin{equation*}
		\tilde{w}_k \stackrel{d}{\rightarrow} \tilde{N}_k,
	\end{equation*}
	where 
	\begin{equation*} % \label{eqn:tilde_N_p}
		\tilde{N}_{k} = N_k + k M_{k-1} a_0,
	\end{equation*}
	with $\{N_k;  k \geq 1\}$  as in Theorem \ref{thm:LES_TT'} and  $M_k = \frac{1}{n}\lim\limits_{n\to\infty} \E \big[\Tr(T_{n \times p}T'_{n \times p})^{k} \big]$, given in (\ref{eqn:2r-Moment_symTnp}).
\end{theorem}
The idea of the proof of Theorem \ref{thm:LES_TT'_a0} is similar to the proof of Theorem \ref{thm:LES_TT'} with some technical changes. We skip the proof. Note from Remark \ref{rem: Mr >0} that $M_k$ is non-zero for all $k \geq 1$. Therefore $\tilde{N}_{k}$ will be Gaussian only when $a_0$ is Gaussian.
	% and \red{independent of $\{a_i\}_{i \geq 1}$}.
	%If $\{a_0;t\ge 0\}$ is a normal distribution, then the limit in Theorem \ref{thm:LES_TT'} will be Gaussian.

\section{\large \bf \textbf{Symmetric Toeplitz with  Brownian motion entries}}\label{sec:T_np_process}
In this section, we consider time-dependent symmetric Toeplitz matrix with  Brownian motion entries ($T_{n\times p}(t)$). In Section \ref{sec:thm_TT'(t)_nonzero}, we  study the process convergence of the linear eigenvalue statistics of $T_{n\times p}(t) T'_{n \times p}(t)$ when diagonal entries are zero. For non-zero diagonal entries, the linear eigenvalue statistics of  $T_{n\times p}(t) T'_{n \times p}(t)$ is discussed in Section \ref{sec:thm_TT'(t)_nonzero}.

\subsection{\large \bf Fluctuation of $T_{n \times p}(t)$  with diagonal entries as zero}  \label{sec:thm_TT'(t)_zero} 
In this section, we prove Theorem \ref{thm:LES_Tp}. We use standard results from process convergence to establish the result. Note that, 
to establish the process convergence of $\{w_{k}(t);t \geq 0\}$, it is sufficient to show the finite dimensional convergence of $\{w_{k}(t);t \geq 0\}$ and the tightness of the process $\{w_{k}(t);t \geq 0\}$. For more details on process convergence, see (Chapter 1.4 ,\cite{Ikeda1981}).
% That is, we have to prove the following two propositions: 

\subsubsection{\large \bf Finite dimensional convergence:} We prove the following proposition.
\begin{proposition} \label{pro:finite convergence}
Let $ k\geq 1$ and $0<t_1<t_2< \cdots <t_r$. Then as $n,p \tends \infty$ and $p/n \rightarrow \lambda$,
	\begin{equation*}
		(w_{k}(t_1), w_{k}(t_2), \ldots , w_{k}(t_r)) \stackrel{d}{\rightarrow} (N_{k}(t_1), N_{k}(t_2), \ldots , N_{k}(t_r)),
	\end{equation*}	
	where $\{N_{k}(t);t\geq 0, k\geq 1\}$ are zero mean Gaussian processes with covariance structure as in \eqref{eqn:covTT'(t)}.
\end{proposition}
Before proving Proposition \ref{pro:finite convergence}, we first derive the limiting covariance structure of $\{w_{k}(t);t \geq 0\}$. 
Recall the relation $\sim_\pi$ defined in Definition \ref{defn:partition}. Observe that $\sim_\pi$ determines the partition uniquely. Using this idea, we now construct a new partition $\pi=(\pi_1,\pi_2)$ from two given partitions $\pi_1$ and $\pi_2$.
\begin{definition}\label{remark:pi1,pi2}
For a partition $\pi_1$ of $[n]$ and a partition $\pi_2$ of $[m]$, we construct the partition $\pi=(\pi_1,\pi_2)$ of $[n+m]$ as the partition with the following block structure.
	\begin{enumerate}[(i)]
		\item 	For $1 \leq i,j \leq n$, $i \sim_\pi j$ if $i \sim_{\pi_1} j$.
		\item For $n+1 \leq i,j \leq n+m$, $i \sim_\pi j$ if $(i-n) \sim_{\pi_2} (j-n)$.
		\item For $1 \leq i \leq n < j \leq n+m$, $i \nsim_\pi j$.
	\end{enumerate}
\end{definition}
\begin{example}
	%\red{Rewrite this, with details and recalling all the notation}
	Let $r \in \{0,2, 4, \ldots, 2k_2\}$ be a fixed integer. Note that for the partition $\pi=(\pi_1,\pi_2)$, $i \sim_\pi j$ if and only if $i \sim_{\pi_1} j$ or $(i-n) \sim_{\pi_2} (j-n)$. Thus for $(\pi_1,\pi_2) \in \PP_{2,4}(2k_1,2r) \times \PP_2(2(k_2+r))$, the partition $\pi=(\pi_1,\pi_2)$ belongs to $\PP_{2,4}(2k_1,2k_2)$ and for $(\pi_1,\pi_2) \in \PP_{2}(2k_1,2r) \times \PP_2(2(k_2+r))$, the partition $\pi=(\pi_1,\pi_2)$ belongs to $\PP_{2}(2k_1,2k_2)$. 
\end{example}
The following lemma provides the limiting covariance structure of $\{w_{k}(t);t \geq 0\}$.
\begin{lemma} \label{lem:covarTT'(t)}
	 Let $\{b_i\}_{i \in \N}$ be an independent sequence of Brownian motions and let $b_0 \equiv 0$. Suppose $T_{n \times p}$ is the $n \times p$ time-dependent symmetric Toeplitz matrix with input sequence $\{\frac{b_i}{\sqrt{n}}\}_{i \geq 0}$ and let $w_{k}(t)$ be as in (\ref{eqn:w_p(t)_Tp}). Then for $0<t_1\leq t_2$ and  positive integers $k_1$, $k_2$,
	\begin{align}\label{eqn:covTT'(t)}
		& \lim_{n \rightarrow \infty \atop p/n \rightarrow \lambda>0} \Cov(w_{k_1}(t_1), w_{k_2}(t_2))  \ = \sum_{r=1}^{k_2} \binom{2k_2}{2r} t_1^{{k_1+r}} (t_2-t_1)^{{k_2-r}} \mathcal{R}(p,r),
	\end{align}
%where $M_{2(k_2-r/2)}$ is as in (\ref{eqn:2r-Moment_symTnp}) and
%\begin{align*}
%\theta_1(k_1,r)	&=\sum_{\pi \in \PP_2(2k_1,r)}  (f^{-}_I(\pi)+f^{+}_I(\pi)) + 2 \sum_{\pi \in \PP_{2,4}(2k_1,r)}  (f^-_{II}(\pi)+f^+_{II}(\pi)),
%%\theta_2(r,k_2)	&=\sum_{\pi \in \PP_2(2k_1,r)}  (f^{-}_I(\pi)+f^{+}_I(\pi)) + 2 \sum_{\pi \in \PP_{2,4}(2k_1,r)}  (f^-_{II}(\pi)+f^+_{II}(\pi)),
%\end{align*}
	where 
	$$\mathcal{R}(p,r)=\sum_{\pi=(\pi_1,\pi_2) \in \atop \PP_2(2k_1,2r) \times \PP_2(2(k_2-r))}  (f^{-}_I(\pi)+f^{+}_I(\pi)) + 2 \sum_{\pi =(\pi_1,\pi_2) \in \atop \PP_{2,4}(2k_1,2r) \times \PP_2(2(k_2-r))}  (f^-_{II}(\pi)+f^+_{II}(\pi)),$$ 
	with the notion $(\pi_1,\pi_2) \in \PP_2(2k_1,2r) \times \PP_2(2(k_2-r))$ as in Definition \ref{remark:pi1,pi2}; $f^-_I(\pi),f^+_I(\pi), f^-_{II}(\pi)$ and $f^+_{II}(\pi)$ as in (\ref{eqn:f_Ipi-_T}), (\ref{eqn:f_Ipi+_T}),  (\ref{eqn:f_II-pi_T}) and (\ref{eqn:f_II+pi_T}), respectively.
\end{lemma}

\begin{proof}
First we define $U:=T_{n \times p}(t_1)$ and $V:=T_{n \times p}(t_2)-T_{n \times p}(t_1)$, where $T_{n \times p}(t_1)$ and $T_{n \times p}(t_2)$ are the time-dependent symmetric Toeplitz matrices with Brownian motion entries at times $t_1$ and $t_2$, respectively. Further for $0 \leq r \leq k_2$, we define
\begin{align} \label{eqn:J,J',U_j}
	&J =(j_1,j_2, \ldots, j_{2k_1}), \ J^{\prime} = (j_1^{\prime},j_2^{\prime},\ldots ,j_{2k_2}^{\prime}),	\
	J_{1}^{\prime}=\left(j_{1}^{\prime},j_2^{\prime}, \ldots, j_{2r}^{\prime}\right),  \ J_{2}^{\prime}=\left(j_{2r+1}^{\prime}, \ldots, j_{2k_2}^{\prime}\right), \\  
	&u_{J}   = u_{j_{1}} \cdots u_{j_{2k_1}}, \  u_{J_{1}^{\prime}}=u_{j_{1}^{\prime}} \cdots u_{j_{2r}^{\prime}} \ \mbox{ and } \ v_{J_{2}^{\prime}}=v_{j_{2r+1}^{\prime}} \cdots v_{j_{2k_2}^{\prime}},\nonumber
\end{align} 
where  $u_{r}=b_{r}(t_1)$ and  $v_{r}=b_{r}(t_2)-b_{r}(t_1)$.	
Now with the above notations and Lemma \ref{lem:trac_TT^t}, $\Cov\left(w_{k_1}\left(t_{1}\right) w_{k_2}\left(t_{2}\right)\right)$ can be written as
\begin{align*} %\label{eqn: cov wp(t) wq(t)}
	&\Cov\left(w_{k_1}\left(t_{1}\right),w_{k_2}\left(t_{2}\right)\right) \nonumber \\
	&=\frac{1}{n^{k_1+k_2+1}} \sum_{r=0}^{2k_2}
	{2k_2 \choose r} \sum_{i,i^\prime=1}^n
	\sum_{J, J^{\prime}}\Big\{\left(\E[u_{J} u_{J_{1}^{\prime}}]-\E[u_{J}] \E[u_{J_{1}^{\prime}}]\right) \E[v_{J_{2}^{\prime}}]\Big\} g(i,i',J,J^\prime,p,n),
\end{align*}
where $J \in A_{k_1}$ and $J^{\prime} \in A_{k_2}$ and $g(i,i',J,J^\prime,p,n)$ is as in (\ref{eqn:g(J,Jprime,p,n)}). Now for a fixed $r$, $0 \leq r \leq 2k_2$, consider 
\begin{align}\label{eqn:J,J_1,J_2}
	& \frac{1}{n^{k_1+k_2+1}} \sum_{i,i^\prime=1}^n \sum_{J, J^{\prime}}\Big\{\left(\E[u_{J} u_{J_{1}^{\prime}}]-\E[u_{J}] \E[u_{J_{1}^{\prime}}]\right) \E[v_{J_{2}^{\prime}}]\Big\}g(i,i',J,J^\prime,p,n) \nonumber\\
	 &=\frac{1}{n^{k_1+k_2+1}} \sum_{\pi}\sum_{i,i^\prime=1}^n \sum_{(J, J^{\prime})\in \Pi(\pi)}\Big\{\left(\E[u_{J} u_{J_{1}^{\prime}}]-\E[u_{J}] \E[u_{J_{1}^{\prime}}]\right) \E[v_{J_{2}^{\prime}}]\Big\}g(i,i',J,J^\prime,p,n),
\end{align}
where the summation $\pi$ is over all partitions of $[2(k_1+k_2)]$ and $\Pi(\pi)$ is as defined in (\ref{eqn: Cov_TT^t one}).

Consider the term
$\left(\E[u_{J} u_{J_{1}^{\prime}}]-\E[u_{J}] \E[u_{J_{1}^{\prime}}]\right) \E[v_{J_{2}^{\prime}}]$. Due to the independence of random variables $\{u_i\}$ and $\{ v_i \}$, this term is non-zero only when the following three conditions are satisfied:

\begin{enumerate}[(a)]
	\item there exists at least one common element between $S_J$ and $S_{J_1^{\prime}}$, that is, $S_J \cap S_{J_1^{\prime}} \neq \emptyset$,
	\item every element of $S_J \cup S_{J_1^{\prime}}$ has cardinality at least 2, and
	\item every element of $S_{J_2^{\prime}}$ has cardinality at least 2,
\end{enumerate}
where the multi-sets $S_{J},S_{J_1^\prime},$ and $S_{J_2^\prime}$ are as defined in (\ref{def:S_|J|}). 

From the proof of Theorem \ref{thm:covarTT'}, it follows that the limit of  (\ref{eqn:J,J_1,J_2}) is non-zero only when the partition $\pi$ belongs to $\PP_{2,4}(2k_1,2k_2)$ or $\pi \in \PP_2(2k_1,2k_2)$. We now find the contribution corresponding to  $\pi \in \PP_{2,4}(2k_1,2k_2)$ and $\pi \in \PP_2(2k_1,2k_2)$ separately. 

First consider the case when $\pi \in \PP_{2,4}(2k_1,2k_2)$ with  $|V_i|=4$. By condition (a), for non-zero contribution in the limit, two elements of $V_i$ must belong to $\{1,2,\ldots , 2k_1\}$ and the other two elements of $V_i$ must belong to $\{2k_1+1,\ldots, 2k_1+r\}$. Note that in this case, conditions (b) and (c) are satisfied only if $r$ is even. Further considering $\pi=(\pi_1,\pi_2)$ as defined in Definition \ref{remark:pi1,pi2}, we get that (\ref{eqn:J,J_1,J_2}) is non-zero only if $\pi_1 \in \PP_{2,4}(2k_1,r)$ and $\pi_2 \in \PP_{2}(2k_2-r)$. By the properties of Brownian motion, $\E u_j^2=t_1, \E v_j^2=(t_2-t_1)$ and $\E u_j^4=3t_1^2$, and in this case
\begin{equation}\label{eqn:u_to_a_2,4}
	\left(\E[u_{J} u_{J_{1}^{\prime}}]-\E[u_{J}] \E[u_{J_{1}^{\prime}}]\right) \E[v_{J_{2}^{\prime}}]=t_1^{k_1+\frac{r}{2}}(t_2-t_1)^{k_2-\frac{r}{2}} 	\left(\E[a_{J} a_{J_{1}^{\prime}}]-\E[a_{J}] \E[a_{J_{1}^{\prime}}]\right) \E[a'_{J_{2}^{\prime}}],
\end{equation}
where  $a_J,a_{J_1^{\prime}}$ and $a'_{J_2^{\prime}}$ are the appropriate product of $a_i,a'_i$ which have the following properties:
 $a_i \distas{D} \frac{1}{\sqrt{t_1}}u_i$, $a'_i \distas{D} \frac{1}{\sqrt{t_2-t_1}}v_i$; $a_i,a'_i$ obeys Assumption I and $\E a_i^4=3$. Here the notion $x_1 \distas{D} x_2$ denotes that  $x_1$ and $x_2$ have same distribution.  From the proof of Theorem \ref{thm:covarTT'} and (\ref{eqn:u_to_a_2,4}), it follows that the summand in (\ref{eqn:J,J_1,J_2}) converges to $2t_1^{k_1+\frac{r}{2}}(t_2-t_1)^{k_2-\frac{r}{2}}\left(f_{II}^-(\pi)+f_{II}^+(\pi)\right)$.

Now consider $\pi \in \PP_{2}(2k_1,2k_2)$. Note that condition (c) implies that the summand in (\ref{eqn:J,J_1,J_2}) is non-zero only if $r$ is even. Furthermore, by condition (c), we get that there is non-zero contribution in limit only if no block of $\pi$ intersects with both $\{1,2,\ldots,(2k_1+r)\}$ and $\{(2k_1+r+1),\ldots
,2k_2\}$. Considering $\pi=(\pi_1,\pi_2)$ as defined in Definition \ref{remark:pi1,pi2}, we get that $\pi_1 \in \PP_{2}(2k_1,r)$ and $\pi_2 \in \PP_{2}(2k_2-r)$. Thus, we get that 
\begin{equation}
	\left(\E[u_{J} u_{J_{1}^{\prime}}]-\E[u_{J}] \E[u_{J_{1}^{\prime}}]\right) \E[v_{J_{2}^{\prime}}]=t_1^{k_1+\frac{r}{2}}(t_2-t_1)^{k_2-\frac{r}{2}} 	\left(\E[a_{J} a_{J_{1}^{\prime}}]-\E[a_{J}] \E[a_{J_{1}^{\prime}}]\right) \E[a'_{J_{2}^{\prime}}],
\end{equation}
where $a_J,a_{J_1^{\prime}}$ and $a'_{J_2^{\prime}}$ are as in (\ref{eqn:u_to_a_2,4}). Again, from the proof of Theorem \ref{thm:covarTT'}, it follows that the summand in (\ref{eqn:J,J_1,J_2}) converge to $t_1^{k_1+\frac{r}{2}}(t_2-t_1)^{k_2-\frac{r}{2}}\left(f_{I}^-(\pi)+f_{I}^+(\pi)\right)$. On combining both the cases, we get (\ref{eqn:covTT'(t)}). This completes the proof of the lemma.
\end{proof}

\begin{proof}[Proof of Proposition \ref{pro:finite convergence}] 
We use the Cram\'er-Wold theorem, Wick formula and the method of moments to prove Proposition \ref{pro:finite convergence}. Note that it is enough to show that, for $0<t_1 \leq t_2 \leq \cdots \leq t_\ell $ and $k_1, k_2, \ldots , k_\ell \geq 1$,
\begin{equation}\label{eqn:limE[w1+k(t)]}
\lim_{n \rightarrow \infty \atop p/n \rightarrow \lambda>0}\E[w_{{k}_1}(t_1) w_{{k}_2}(t_2) \cdots w_{{k}_\ell}(t_\ell)]=\E[N_{{k}_1}(t_1)N_{{k}_2}(t_2) \cdots N_{{k}_\ell}(t_\ell)],
\end{equation}
where $\{N_{k}(t);t\geq 0, k\geq 1\}$ are zero mean Gaussian processes with covariance structure as in \eqref{eqn:covTT'(t)}.
The idea of the proof of (\ref{eqn:limE[w1+k(t)]}) is similar to the proof of (\ref{eqn:limE[w1+k]}), we skip it.
\end{proof}

\subsubsection{\large \bf Tightness:} To establish the tightness of the process $\{w_{k}(t);t \geq 0\}$, we check the conditions of the following Proposition which is a sufficient condition for the tightness of $\{w_{k}(t);t \geq 0\}$. 
\begin{proposition} (Theorem I.4.3, \cite{Ikeda1981}) \label{pro:tight}
	For each $k \geq 1$, there exists positive constants $M$ and $\gamma$ such that
	\begin{equation} \label{eqn:tight3}
		\E{|w_{k}(0)|^ \gamma } \leq M \ \ \ \forall \ n \in \mathbb{N}, 	 
	\end{equation}
	and there exists positive constants $\alpha, \ \beta $ and $M_T$, $T= 1, 2, \ldots,$ such that
	\begin{equation} \label{eq:tight4}
		\E{|w_{k}(t)- w_{k}(s)|^ \alpha } \leq M_T |t-s|^ {1+ \beta} \ \ \ \forall \ n \in \mathbb{N} \ \mbox{and } t,s \in[0,\ T],(T= 1, 2, \ldots,).
	\end{equation}
%	Then $\{ w_{k}(t) ; t \geq 0 \}$ is tight.
\end{proposition} 
\begin{proof}[Proof of Proposition \ref{pro:tight}] 
	First recall $w_{k}(t)$ from (\ref{eqn:w_p(t)_Tp}),
	$$ w_{k}(t) = \frac{1}{\sqrt{n}} \bigl\{ \Tr(X_n(t))^{k} - \E[\Tr(X_n(t))^{k}]\bigr\},$$
	where $X_n(t)=T_{n \times p}(t) T'_{n \times p}(t)$.
	Note that $w_{k}(0)= 0$ for all $k \geq 1$. This shows that (\ref{eqn:tight3}) trivially follows.
	
	Now we  prove (\ref{eq:tight4}) for $\alpha = 4$ and $\beta = 1$. 
	Suppose $k \geq 2$ is fixed and $t,s \in[0,T] $, for some fixed $T \in \mathbb{N}$. Then
	\begin{align} \label{eqn:w_p(t-s)}
		w_{k}(t) - w_{k}(s) &=\frac{1}{\sqrt{n}} \big[ \Tr(X_n(t))^{k} - \Tr(X_n(s))^{k} - \E[\Tr(X_n(t))^{k} - \Tr(X_n(s))^{k}]\big].
		%  \E[w_{2p}(t) -w_{p}(s)]^4 &= \frac{1}{n^2} \E \big[ \Tr(T_n(t))^{2p} - \Tr(RC_n(s))^{2p} - \E[\Tr(RC_n(t))^{2p} - \Tr(RC_n(s))^{2p}]\big]^4.
	\end{align}
	For $0< s <t $, we have
	\begin{align*}
		X_n(t) &= X_n(t) - X_n(s) + X_n(s)
		= \widetilde{X}_n(t-s) + X_n(s),
	\end{align*}
	where $\widetilde{X}_n(t-s) := T_{n \times p}(t) T'_{n \times p}(t)- T_{n \times p}(s) T'_{n \times p}(s)$. 
%	\begin{equation*}
%		(X_n(t))^{k}  = [X'_n(t-s)+ X_n(s)]^{k}.
%	\end{equation*} 
	Now using the binomial expansion, we get 
	\begin{align} \label{eqn:C_n^p(t)- C_n^p(s)}
		(X_n(t))^{k} - (X_n(s))^{k} & = (\widetilde{X}_n(t-s))^{k} + \sum_{d=1}^{k-1} \binom{k}{d} (\widetilde{X}_n(t-s))^d (X_n(s))^{k-d}, \nonumber 
	\end{align} 
	and hence
	\begin{equation} \label{eqn:trace t-s}
		\Tr (X_n(t))^{k} - \Tr (X_n(s))^{k}  = \Tr[ (\widetilde{X}_n(t-s) )^{k}] + \sum_{d=1}^{k-1} \binom{k}{d} \Tr [(\widetilde{X}_n(t-s))^d (X_n(s))^{k-d}].
	\end{equation}
	Now, we calculate each term of the above expression. First, using the trace formula (\ref{def:trace_Tn_A}), we get 
	\begin{align} \label{eqn:tracemultply}
		&\Tr[ (\widetilde{X}_n(t-s) )^{k}] \nonumber \\ & = \frac{1}{n^{k} } \sum_{i=1}^n \sum_{A_{k}} \prod_{r=1}^{k} \big\{b_{j_{2r-1}}(t) b_{j_{2r}}(t) -b_{j_{2r-1}}(s) b_{j_{2r}}(s) \big\} I(i, J_p,p,n) \nonumber \\
		 & =  \frac{1}{n^{k} } \sum_{i=1}^n \sum_{A_{k}} \prod_{r=1}^{k} \big\{ b_{j_{2r-1}}(t) b_{j_{2r}}(t) - b_{j_{2r-1}}(t) b_{j_{2r}}(s) + b_{j_{2r-1}}(t) b_{j_{2r}}(s) -b_{j_{2r-1}}(s) b_{j_{2r}}(s) \big\} I(i, J_k,p,n)   \nonumber \\
		 & =  \frac{1}{n^{k} } \sum_{i=1}^n \sum_{A_{k}} \prod_{r=1}^{k} \big[ b_{j_{2r-1}}(t) \big\{b_{j_{2r}}(t) - b_{j_{2r}}(s) \big\} + \big\{b_{j_{2r-1}}(t) - b_{j_{2r-1}}(s) \big\} b_{j_{2r}}(s) \big] I(i, J_k,p,n)   \nonumber \\
		  & \distas{D}  \frac{1}{n^{k} } \sum_{i=1}^n \sum_{A_{k}} \prod_{r=1}^{k} \big[  \sqrt{t} y_{j_{2r-1}}  \sqrt{t-s} x_{j_{2r}} + \sqrt{t-s}x_{j_{2r-1}} \sqrt{s} y_{j_{2r}}\big] I(i, J_k,p,n)  \nonumber \\
		  & = \frac{(t-s)^\frac{k}{2}}{n^{k} } \sum_{i=1}^n \sum_{A_{k}} \prod_{r=1}^{k} \big\{ \sqrt{t} y_{j_{2r-1}} x_{j_{2r}} + \sqrt{s} x_{j_{2r-1}}  y_{j_{2r}} \big\} I(i, J_k,p,n) ,
	\end{align}
	where the notion $x_1 \distas{D} x_2$ denotes that  $x_1$ and $x_2$ have same distribution, and for any $j \in \N$ and $t,s \in (0,T)$, 
	$$ \text{$b_j(t-s) \distas{D} (\sqrt{t-s})x_j$, $b_j(s) \distas{D} \sqrt{s} y_j$ and $b_j(t) \distas{D} \sqrt{t} y_j$}.$$
	Since $\{b_j(t); t \geq 0\}$ is a Brownian motion, $\{x_{j}\}$ and $\{y_{j}\}$ are independent normal random variables with mean zero, variance 1 and other moments finite.

	Similarly, using the trace formula (\ref{def:trace_Tn_A}), we get
		\begin{align} \label{eqn:tracemultply(t-s)s}
		& \Tr [(\widetilde{X}_n(t-s))^d (X_n(s))^{k-d}] \nonumber \\
		 & = \frac{1}{ n^k } \sum_{i=1}^n \sum_{A_{k}}   \prod_{r=1}^{k-d} (b_{j_{2r-1}}(s) b_{j_{2r}}(s) \prod_{r=(k-d)+1}^{k} \big\{b_{j_{2r-1}}(t) b_{j_{2r}}(t) -b_{j_{2r-1}}(s) b_{j_{2r}}(s) \big\}  I(i, J_k,p,n)  \nonumber \\
		& \distas{D} \frac{(t-s)^\frac{d}{2}}{n^{k} } \sum_{i=1}^n \sum_{A_{k}} \prod_{r=1}^{k-d} (s)^{\frac{k-d}{2}} y_{j_{2r-1}} y_{j_{2r}}  \prod_{r=(k-d)+1}^{k} \big\{ \sqrt{t} y_{j_{2r-1}} x_{j_{2r}} + \sqrt{s} x_{j_{2r-1}}  y_{j_{2r}} \big\}   I(i, J_k,p,n) .
	\end{align}
	
	Now using (\ref{eqn:tracemultply}) and (\ref{eqn:tracemultply(t-s)s}) in (\ref{eqn:trace t-s}), we get 
	\begin{align} \label{eqn:Z_p}
		&\Tr (X_n(t))^{k} - \Tr (X_n(s))^{k} \nonumber \\
		& \distas{D}  \frac{1}{n^{k} } \sum_{i=1}^n \sum_{A_{k}} \Big( (t-s)^\frac{k}{2} \prod_{r=1}^{k} \big\{ \sqrt{t} y_{j_{2r-1}} x_{j_{2r}} + \sqrt{s} x_{j_{2r-1}}  y_{j_{2r}} \big\} \nonumber \\
		& \quad + \sum_{d=1}^{k-1} (t-s)^\frac{d}{2} \prod_{r=1}^{k-d} (s)^{\frac{k-d}{2}} y_{j_{2r-1}} y_{j_{2r}}  \prod_{r=(k-d)+1}^{k} \big\{ \sqrt{t} y_{j_{2r-1}} x_{j_{2r}} + \sqrt{s} x_{j_{2r-1}}  y_{j_{2r}} \big\} \Big)  I(i, J_k,p,n)  \nonumber \\
		& =  \frac{\sqrt{t-s}}{n^{k} } \sum_{i=1}^n \sum_{A_{k}} \Big( (t-s)^\frac{k-1}{2} \prod_{r=1}^{k} \big\{ \sqrt{t} y_{j_{2r-1}} x_{j_{2r}} + \sqrt{s} x_{j_{2r-1}}  y_{j_{2r}} \big\} \nonumber \\
		& \quad + \sum_{d=1}^{k-1}(t-s)^\frac{d-1}{2} \prod_{r=1}^{k-d} (s)^{\frac{k-d}{2}} y_{j_{2r-1}} y_{j_{2r}}  \prod_{r=(k-d)+1}^{k} \big\{ \sqrt{t} y_{j_{2r-1}} x_{j_{2r}} + \sqrt{s} x_{j_{2r-1}}  y_{j_{2r}} \big\} \Big)  I(i, J_k,p,n)  \nonumber \\
		& = \frac{\sqrt{t-s}}{n^{k} } \sum_{i=1}^n \sum_{A_{k}} Z_{J_{k}} I(i, J_k,p,n), \mbox{ say}.
	\end{align} 
	Finally,  using (\ref{eqn:Z_p}) in (\ref{eqn:w_p(t-s)}), we get
	\begin{align*} % \label{eqn:w_p(t-s)2}
		w_{k}(t) -w_{k}(s) \distas{D} \frac{\sqrt{t-s}}{n^{k+ \frac{1}{2}} } \sum_{i=1}^n \sum_{A_{k}} (Z_{J_{k}} - \E[Z_{J_{k}}]) I(i, J_k,p,n),
	\end{align*}
	and hence 
	\begin{align} \label{eqn:w_p^4}
		\E[w_{k}(t) -w_{k}(s)]^4 & = \frac{(t-s)^2}{n^{4k+2}} \sum_{i_1, i_2, i_3, i_4=1}^n  \sum_{A_{k}, A_{k}, A_{k}, A_{k} } \E \Big[ \prod_{m=1}^4 (Z_{J^m_{k}} - \E Z_{J^m_{k}})  I(i_m, J^m_k,p,n)\Big], 
	\end{align}
	where $J^m_{k}$ are vectors from $A_{k}$, for each $m=1, 2, 3, 4$. Depending on connectedness between $J^m_{k}$'s, 
	%we divide the right side of \eqref{eqn:w_p^4} into three parts, $I_1,I_2, I_3$. 
	the following three cases arise:
	\vskip5pt
	\noindent \textbf{Case I.} \textbf{At least one of $\{J^m_{k}: m=1,2,3,4\}$, is not connected with the remaining ones:}  In this case, due to independence of entries, we get
	$$\E \Big[ \prod_{m=1}^4 (Z_{J^m_{k}} - \E Z_{J^m_{k}}) \Big] = 0.$$
	Hence in this case, contribution to $\E[w_{p}(t) -w_{p}(s)]^4$ is zero.
	%we always get $\E[w_{p}(t) -w_{p}(s)]^4 = 0$ which shows that (\ref{eq:tight4}) is true. \\\
	\vskip5pt
	\noindent \textbf{Case II.} \textbf{$J^1_{k}$ is connected with only one of $J^2_{k}, J^3_{k}, J^4_{k}$ and the remaining two of $J^2_{k}, J^3_{k}, J^4_{k}$  are connected only among themselves:} Without loss of generality, we assume $J^1_{k}$ is connected with $J^2_{k}$ and $J^3_{k}$ is connected with $J^4_{k}$. 
	%that is, $(S_{J^1_{p}} \cup S_{J^2_{p}}) \cap (S_{J^3_{p}} \cup S_{j^4_{p}}) = \emptyset$. 
	Under this situation, the terms in the right hand side of (\ref{eqn:w_p^4}) can be written as
	\begin{align} \label{eqn:case II} 
		&\frac{(t-s)^2}{n^{4k+2}} \hskip-5pt \sum_{i_1, i_2, i_3, i_4=1}^n  \sum_{A_{k}, A_{k}} \hskip-4pt \E \Big[ \prod_{m=1}^2 (Z_{J^m_{k}} - \E Z_{J^m_{k}}) I(i_m, J^m_k,p,n) \Big]  \sum_{A_{k}, A_{k}} \hskip-4pt \E \Big[ \prod_{m=3}^4 (Z_{J^m_{k}} - \E Z_{J^m_{k}})  I(i_m, J^m_k,p,n)\Big]  \nonumber \\
		&= T_2, \mbox{ say}.
	\end{align}
	First note that
	\begin{align} \label{eqn:B_p_2}
	\sum_{A_{k}, A_{k}} \E \Big[ \prod_{m=1}^2 (Z_{J^m_{k}} - \E Z_{J^m_{k}}) \Big] = \sum_{( J^1_{k}, J^2_{k}) \in B_{K_2}} \E \Big[ \prod_{m=1}^2 (Z_{J^m_{k}} - \E Z_{J^m_{k}}) \Big],
	\end{align}
	where $B_{K_2}$ as in Lemma \ref{lem:B_{P_l}}.
	
	Since $x_i, y_i$ are normal random variables and $t, s \in [0,T]$,  there exists $\alpha>0$ such that  for all $( J^1_{k}, J^2_{k}) \in B_{K_2}$,
	$$\big| \E \Big[ \prod_{m=1}^2 (Z_{J^m_{k}} - \E Z_{J^m_{k}}) \Big] \prod_{m=1}^2 I(i_m, J^m_k,p,n) \big| \leq \alpha. $$
	Now using the above bound, from (\ref{eqn:B_p_2}) we get
	\begin{align*} % \label{eqn:case2, B_2}
		\sum_{ A_{k}, A_{k}} \Big|  \E \Big[ \prod_{m=1}^2 (Z_{J^m_{k}} - \E Z_{J^m_{k}}) \Big] \prod_{m=1}^2 I(i_m, J^m_k,p,n)  \Big| \leq \sum_{( J^1_{k}, J^2_{k}) \in B_{K_2}} \alpha = |B_{K_2}| \times \alpha.
	\end{align*} 
Note from  Lemma \ref{lem:B_{P_l}} that $|B_{K_2}|= O({n}^{2k-1})$ and hence from the above expression and (\ref{eqn:case II}),
	\begin{align*}
		%\E[w_{p}(t) -w_{p}(s)]^4 
		|T_2| \leq \frac{(t-s)^2}{ n^{4}} \sum_{i_1, i_2, i_3, i_4=1}^n  \alpha^2 O(1) = (t-s)^2 \alpha^2 O(1).
	\end{align*}
%	Since $t, s \in [0,T]$,  there exist $M_1 >0$, depending only on $p, T$ such that 
%	\begin{align} \label{case II,(t-s)}
%		%\E[w_{p}(t) -w_{p}(s)]^4 
%		|T_2|\leq M_1 (t-s)^2 \ \mbox{ for all }n\geq 1.
%	\end{align} 
\vskip3pt
	\noindent \textbf{Case III.} \textbf{$\{J^1_{k}, J^2_{k}, J^3_{k}, J^4_{k}\}$ forms a cluster:}  
	Since $\E(x_{i})=0$ and $\E(y_{i})=0$, in this case we have 
	\begin{align} \label{eqn:B_p_2,2}
		 \sum_{A_{k}, A_{k}, A_{k}, A_{k}} \E \Big[ \prod_{m=1}^4 (Z_{J^m_{k}} - \E Z_{J^m_{k}}) \Big]
		= \sum_{( J^1_{k}, J^2_{k}, J^3_{k}, J^4_{k}) \in B_{K_4}}  \E \Big[ \prod_{m=1}^4 (Z_{J^m_{k}} - \E Z_{J^m_{k}}) \Big],
	\end{align}
	where $B_{K_4}$ as in Lemma \ref{lem:B_{P_l}} for $\ell=4$ and $k_i=k$, $i=1,2,3,4$. Again by the similar arguments as given  in Case II, there exists $\beta>0$ such that for each $( J^1_{k}, J^2_{k}, J^3_{k}, J^4_{k}) \in B_{K_4}$,
	$$\Big| \E \Big[ \prod_{m=1}^4 (Z_{J^m_{k}} - \E Z_{J^m_{k}}) \Big] I(i_m, J^m_k,p,n) \Big| \leq \beta. $$
	Using the above inequality, from (\ref{eqn:B_p_2,2}) we get
	\begin{align*}
		\sum_{ A_{k}, A_{k}, A_{k}, A_{k}} \Big| \E \Big[ \prod_{m=1}^4 (Z_{J^m_{k}} - \E Z_{J^m_{k}}) \Big] I(i_m, J^m_k,p,n) \Big| & \leq \sum_{( J^1_{k}, J^2_{k}, J^3_{k}, J^4_{k}) \in B_{K_4}} \beta
		=|B_{K_4}| \times \beta.
	\end{align*} 
	Recall from Lemma \ref{lem:B_{P_l}} that $|B_{K_4}| = o({(n+p)}^{4k-2})$. For $p,n \rightarrow \infty$ and $p/n \tends \lambda>0$, we get
	\begin{align*}% \label{eqn:case III,Z}
		 \frac{(t-s)^2}{n^{4k+2}} \sum_{i_1, i_2, i_3, i_4=1}^n  \sum_{A_{k}, A_{k}, A_{k}, A_{k} } \Big|  \E \Big[ \prod_{m=1}^4 (Z_{J^m_{k}} - \E Z_{J^m_{k}}) \Big] I(i_m, J^m_k,p,n) \Big| = (t-s)^2  o(1).
	\end{align*}

	Combining all three cases, we get that there exists a positive constant $M_T$, depending only on $p, T$ such that
	\begin{align*} % \label{eqn:M_T}
		%\E[w_p(t) - w_p(s)]^4 & \leq M^T_7 (t-s)^2 + M^T_{10} (t-s)^2 \nonumber \\
		\E[w_{k}(t) -w_{k}(s)]^4 & \leq M_T (t-s)^2 \ \ \ \forall \ n \in \mathbb{N} \ \mbox{and } t,s \in[0,\ T].
	\end{align*}
	This completes the proof of Proposition \ref{pro:tight} with $\alpha = 4$ and $\beta = 1$.
\end{proof}

\subsection{\large \bf Fluctuation of $T_{n \times p}(t)$  with non-zero diagonal entries}  \label{sec:thm_TT'(t)_nonzero} 
	In  Theorem \ref{thm:LES_Tp}, we considered $b_0(t)\equiv 0$. Now we discuss the fluctuation behaviour for a non-zero continuous stochastic process $b_0(t)$. First, define the following notion:
	
	Let $\{b_0(t);t\ge 0\}$ be a non-zero continuous stochastic process  which is independent of the Brownian motion sequence $\{b_m(t);t\geq 0, m \geq 1\}$. For $k \geq 1$ we define,  
	\begin{align*} %\label{eqn:w_p(t)_tilde}
		\tilde{T}_{n \times p}(t) & := T_{n \times p}(t) + \frac{b_0(t)}{\sqrt{n}} I_{n \times p}, \nonumber \\
		\tilde{w}_{k}(t) & := \frac{1}{\sqrt{n}}\bigl[ \Tr\big\{ \tilde{T}_{n \times p}(t) \tilde{T}'_{n \times p}(t)\big\}^{k} - \E[\Tr\big\{ \tilde{T}_{n \times p}(t) \tilde{T}'_{n \times p}(t)\big\}^{k}]\bigr],
	\end{align*}
where $T_{n \times p}(t)$ is the Toeplitz matrix whose entries are $\{\frac{b_m(t)}{\sqrt{n}}\}_{m \geq 0}$  with zero diagonal entries ($b_0(t) \equiv 0$). Then for every  $k\geq 1$, as $ n,p \tends \infty$ with $p/n \tends \lambda>0$, we have the following results:
	
\begin{theorem}\label{thm:LES_Tp_b0}
		\noindent (i) For $k \geq 1, r \geq 1$ and $0<t_1 <t_2 <\cdots <t_r$, $n,p \rightarrow \infty$ and $p/n \rightarrow \lambda>0$, we have
	\begin{equation*}
		(\tilde{w}_{k}(t_1),  \ldots , \tilde{w}_{k}(t_r)) \stackrel{d}{\rightarrow} (\tilde{N}_{k}(t_1),  \ldots ,\tilde{N}_k(t_r)),
	\end{equation*}
	where 
	\begin{equation} \label{eqn:tilde_N_p(t)}
		\tilde{N}_{k}(t) = N_k(t) + k t^{k/2} M_{k-1} b_0(t),
	\end{equation}
	with $\{N_k(t);  k \geq 1\}$  as in Theorem \ref{thm:LES_Tp} and  $M_k = \lim\limits_{n\to\infty} \frac{1}{n}\E \big[\Tr\big(\tilde{T}_{n \times p}(1)\tilde{T}'_{n \times p}(1)\big)^{k} \big]$, as given in Theorem \ref{thm: Moment TT^t}.
	Note that $\tilde{N}_{k}(t)$ will be Gaussian only when $b_0(t)$ is a Gaussian and independent of $\{b_m(t)\}_{m \geq 1}$.
	\vskip4pt
	\noindent (ii) The process $\{\tilde{w}_{k}(t);t \geq 0\}$ will be tight only when $b_0(t)$ is a Brownian motion and independent of $\{b_m(t)\}_{m \geq 1}$.
\end{theorem}
For the proof of Theorem \ref{thm:LES_Tp_b0}, we refer the reader to the proof of Theorem \ref{thm:LES_Tp}. One can also see the proof of Theorem 4 of \cite{m&s_toeplitz_2020}.
	 Observe that in Theorem \ref{thm:LES_Tp_b0} (ii), if $b_0(t)$ is not a Brownian motion, then we do not have the tightness of  $\{\tilde{w}_{k}(t);t \geq 0\}$. For details, see the proof of Theorem 4 of \cite{m&s_toeplitz_2020}. Hence  we have the process convergence of $\{\tilde{w}_{k}(t);t \geq 0\}$, only when $b_0(t)$ is a Brownian motion which is independent of $\{b_m(t)\}_{m \geq 1}$, and in this case the limit is $\{\tilde{N}_{k}(t);  t \geq 0\}$. %, as in (\ref{eqn:tilde_N_p(t)}).

\section{\large \bf \textbf{Non-symmetric Toeplitz matrices}} \label{sec:non-sym_Tnp} 

In this section we consider non-symmetric Toeplitz matrices, and study their linear eigenvalue statistic.
\subsection{\large \bf Limiting moment sequence:}
The following theorem provides the limiting sequence of  $\E\left[\frac{1}{n}\Tr(T_{n\times p} T'_{n \times p})^r\right]$ for non-symmetric Toeplitz matrix $T_{n \times p}$.
\begin{theorem}\label{thm: Moment TT^t_non-sym}
	Let $\lambda \in (0,\infty)$ and $X_n=T_{n\times p} T'_{n \times p}$, where $T_{n\times p}$ is the $n\times p$ non-symmetric Toeplitz matrix with input sequence $\{\frac{a_i}{\sqrt{n}}\}_{i \in \Z}$, $\{ a_i\}_{ |i| \geq 1 }$ obeying Assumption I and $a_0 \equiv 0$. Then for $r \geq 1$, $\E\left[\frac{1}{n}\Tr X_n^{r}\right]=M'_{r}+o(1/\sqrt{n})$  as $n,p \rightarrow \infty$, $p/n \rightarrow \lambda$, where
	\begin{align} \label{eqn:2r-moment_non-symT}
		M'_{r}=\sum_{\pi \in \mathcal{DP}_2(2 r)} \int_{[0,1] \times [-\lambda,1]^{r}} \prod_{s=1}^{r} \chi_{[0,\lambda]}\left(x_0- \sum_{\ell=i}^{2s-1} \epsilon_\pi(i) x_{\pi(i)}\right)\chi_{[0,1]}\left( x_0- \sum_{i=1}^{2s} \epsilon_\pi(i) x_{\pi(i)}\right)  \prod_{l=0}^{r} \mathrm{~d} x_l,
	\end{align}
	with $\mathcal{DP}_2(2 r)$ is as in Definition \ref{def:SP_2} and $\epsilon_{\pi}$ as given in (\ref{eqn: epsilon map}).
\end{theorem}
\begin{proof}
	The proof of this theorem is similar to the proof Theorem \ref{thm: Moment TT^t} and here, we mention only the key steps. Note that similar calculations imply that 
\begin{equation}\label{eqn:E1/n tr X non-sym}
	\E\left[\frac{1}{n} \Tr\left(X^{ r}\right)\right]=\frac{1}{n^{r+1}} \sum_{i=1}^n \sum_{\pi \in \widehat{\mathcal{P}}(2r)}\sum_{J \in \Pi(\pi)}  \E\left[a_J\right] I(i,J,p,n) \delta_{0} (\sum_{q=1}^{2r}(-1)^{q} j_{q}),
\end{equation} 
where $\widehat{\mathcal{P}}(2r)$ is the set of all partitions of $[2r]$ such that each block has size greater than or equal to two, and $\Pi(\pi)$ is the set of all $J=(j_1,j_2,\ldots , j_{2r}) \in \{-(p-1),\ldots,-1,1,\ldots,(n-1)\}^{2r}$ such that $r \sim_{\pi} s$ if and only if $j_r=j_s$. By an argument similar to the one in proof of Theorem \ref{thm: Moment TT^t}, it follows that a non-zero contribution in limit occurs only for pair-partitions $\pi$. Suppose $\pi$ contains a same-parity block, say $\{u,v\}$ and let $J \in \Pi(\pi)$. Then $(-1)^uj_u=(-1)^vj_v$ and therefore once all $j_q$ except $j_u,j_v$ are chosen, (\ref{eqn:E1/n tr X non-sym}) is non-zero only if $(-1)^uj_u=-\frac{1}{2}  \sum_{q=1; q \neq r,s}^{2r}(-1)^qj_q$. Thus $j_u$ is determined by other $j_q$'s and therefore the number of choices of $J \in \Pi(\pi)$ such that (\ref{eqn:E1/n tr X non-sym}) is non-zero is of the order $O(n^{r-1})$. Thus, contribution of the order $O(1)$ occurs in the limit only when $\pi \in \mathcal{DP}_2(2 r)$ and the contribution of rest of the terms is the order $o(1/\sqrt{n})$. This completes the proof of the Theorem \ref{thm: Moment TT^t_non-sym}.
\end{proof}

\subsection{\large \bf Fluctuation for independent entries:}
The following theorem provides the fluctuation of linear eigenvalue statistics of $T_{n \times p} T'_{n \times p}$.
\begin{theorem}\label{thm:LES_TT'non-sym}
	Let $\lambda \in (0,\infty)$ and $T_{n \times p}$ be the $n \times p$ non-symmetric Toeplitz matrix  with input entries $\{\frac{a_i}{\sqrt{n}}\}_{i \in \Z}$, where $\{a_i\}_{|i|\geq 1}$ satisfies Assumption I and $a_0 \equiv 0$.
	Then as $ n,p \tends \infty$ with $p/n \tends \lambda$,
%	\begin{equation}
%		w_k \stackrel{d}{\rightarrow} N(0,\sigma_k^2),
%	\end{equation}
%where $w_k$ is as in (\ref{eqn:wp_Hn_odd}) for non-symmetric Toeplitz matrix  $T_{n\times p}$ and 	$\sigma_p^{2}$ is as in (\ref{eqn:cov_TT'non-sym}).
%	Moreover, for a given polynomial $Q(x)=\sum_{j=2}^{k}c_{j} x^{j}$ with degree $k\geq 2$, 
	\begin{equation*}
	 \frac{1}{\sqrt{n}} \big[ \Tr Q(T_{n \times p}T'_{n \times p}) - \E[\Tr Q(T_{n \times p}T'_{n \times p})] \big] \stackrel{d}{\rightarrow} N(0,\widehat{\sigma}_Q^2),
	\end{equation*}
where $Q(x)=\sum_{j=0}^{k}c_{j} x^{j}$ is a real polynomial with degree $k\geq 1$ and $\widehat{\sigma}_Q^2=\sum_{j_1,j_2=0}^{k}c_{j_1} c_{j_2} \widehat{\sigma}_{j_1,j_2}$ with  
	\begin{align}\label{eqn:cov_TT'non-sym}
		\widehat{\sigma}_{j_1,j_2}= \sum_{\pi \in \mathcal{DP}_2(2j_1,2j_2)}  f_I^-(\pi) + (\kappa-1) \sum_{\pi \in \mathcal{DP}_{2,4}(2j_1,2j_2)}  f_{II}^-(\pi),
	\end{align}
	where $\mathcal{DP}_2(2j_1,2j_2)$, $\mathcal{DP}_{2,4}(2j_1,2j_2)$ are as given in Definition \ref{def:SP_2}, $f_I^-(\pi)$ and $f_{II}^-(\pi)$ are the integrals given in (\ref{eqn:f_Ipi-_T}) and (\ref{eqn:f_II-pi_T}), respectively.
\end{theorem}
 The proof of Theorem \ref{thm:LES_TT'non-sym} is similar to the proof of Theorem \ref{thm:LES_TT'}, we skip the proof here.
\subsection{\large \bf Fluctuation for time-dependent entries:}

\begin{theorem}\label{thm:LES_Tp(t)_non-sym} 
	Let $ k\geq 1$ and $\lambda \in (0,\infty)$. Suppose $w_k(t)$ is as in (\ref{eqn:w_p(t)_Tp}) for non-symmetric Toeplitz matrix  $T_{n\times p}(t)$ with the Brownian motion entries as non-diagonal and zero as diagonal entries. Then as $ n,p \tends \infty$ with $p/n \tends \lambda$,
	\begin{equation*}
		\{ w_{k}(t) ; t \geq 0\} \stackrel{\mathcal D}{\rightarrow} \{N_{k}(t) ; t \geq 0\},
	\end{equation*}
	where $\{N_{k}(t);t\geq 0, k\geq 1\}$ are zero mean Gaussian processes with covariance structure as: for $t_1 \leq t_2$
		\begin{align*}%\label{eqn:covTT'(t)non-sym}
	 &\Cov(N_{k_1}(t_1),N_{k_2}(t_2)) \nonumber \\
			& \ = \sum_{r=1}^{k_2} \binom{2k_2}{2r} t_1^{k_1+r} (t_2-t_1)^{k_2-r} \Big[ \sum_{\pi=(\pi_1,\pi_2) \in \mathcal{DP}_2(2k_1,2r) \atop \times \mathcal{DP}_2(2k_2-2r)}  f_I^-(\pi) + 2 \sum_{\pi=(\pi_1,\pi_2) \in \mathcal{DP}_{2,4}(2k_1,2r)\atop \times \mathcal{DP}_2(2k_2-2r)}  f_{II}^-(\pi) \Big],
		\end{align*}
%	where $M'_{2(k_2-r/2)}$ is as in (\ref{eqn:2r-moment_non-symT}) and
%\begin{align*}
%\theta'_1(k_1,r)& = \sum_{\pi \in \mathcal{SP}_2(2k_1,r)}  g_I(\pi) + 2 \sum_{\pi \in \mathcal{SP}_{2,4}(2k_1,r)}  g_{II}(\pi),
%\end{align*}
where $f_I^-(\pi)$ and $f_{II}^-(\pi)$ are as in (\ref{eqn:f_Ipi-_T}) and (\ref{eqn:f_II-pi_T}), respectively; 	$\mathcal{DP}_2(2k_1, 2r)$, $\mathcal{DP}_2(2k_2-2r)$ and $\mathcal{DP}_{2,4}(2k_1,2r)$ are as in Definition \ref{def:SP_2} with the notion $(\pi_1,\pi_2) \in  \mathcal{DP}_2(2k_1,2r) \times \mathcal{DP}_2(2k_2-2r)$ as in Definition \ref{remark:pi1,pi2}.
\end{theorem}
The idea of the proof of Theorem \ref{thm:LES_Tp(t)_non-sym} is similar to the proof of Theorem \ref{thm:LES_Tp}, we skip the proof.

\begin{remark}
	In Theorem \ref{thm:LES_TT'non-sym} and Theorem \ref{thm:LES_Tp(t)_non-sym}, we have considered non-symmetric Toeplitz matrices with  zero diagonal entry. For the non-symmetric Toeplitz matrices with non-zero diagonal entries, one can conclude similar results as Theorem \ref{thm:LES_TT'_a0} and Theorem \ref{thm:LES_Tp_b0}.
\end{remark}

\section{\large \bf \textbf{Hermitian Toeplitz matrices}} \label{sec:her_Tnp} 

In this section, we consider Hermitian Toeplitz matrices. We study the fluctuations of linear eigenvalue statistics of $T_{n \times p}T_{n \times p}^*$, where $T_{n \times p}$ is a Hermitian Toeplitz matrix and $T_{n \times p}^*$ denotes the adjoint of $T_{n \times p}$.

Given an input	sequence $\{a_i\}_{i \geq 0}$, we construct a Hermitian Toeplitz matrix, by defining $a_{-i}=a_{i}$ for all $i \in \N$. We consider the following assumption on the input sequence.
\vskip2pt	
\noindent \textbf{Assumption II.} Let $\{a_{j}= x_j+ \mathrm{i} y_j; j  \geq 0\}$ be a sequence of complex random variables, where $\{(x_{j}, y_{j}); j \in \N \}$ is a sequence of independent real random variables with mean $0$. Further,  for all $j\in \N$, 
\begin{equation*}%\label{eqn: condition Hermitian} 
\displaystyle \E(x_{j}^2)=\E(y_{j}^2)=\frac{1}{2}, \ \E(|a_{j}|^4)=\kappa \text{ and }  \sup_{j} \{\E(|x_{j}|^k),  \E(|y_{j}|^k)\} \leq c_k < \infty, \ \text{ for all}\ \ k\geq 1.
\end{equation*}
 
 We first introduce the trace formula for $T^{(r)}_{n\times p} (T^{(r)}_{n \times p})^*$ when $T^{(r)}_{n\times p}$ is a Hermitian Toeplitz matrix.
\begin{lemma} \label{lem:trac_TT^H}
	Suppose  $T^{(r)}_{n\times p}$ are $n\times p$ Toeplitz matrices with complex input sequence $\{ a^{r}_i\}_{ i \in \mathbb{Z} }$ for $r=1,2, \ldots, k$. Let $X^{(r)}=T^{(r)}_{n\times p} (T^{(r)}_{n \times p})^*$. Then
	\begin{align*}% \label{def:trace_Tn_her}
		\Tr( X^{(k)} \cdots X^{(1)})
		= \sum_{i=1}^{n} \sum_{j_{1}, \ldots, j_{2k}=-(p-1)}^{n-1} \prod_{r=1}^{k} \bar{a}^{(r)}_{j_{2r-1}} a^{(r)}_{j_{2r}}  I(i,J_k,p,n)\delta_{0} (\sum_{q=1}^{2k}(-1)^{q} j_{q}),
		%= \sum_{i=1}^{n} \sum_{A_k} \prod_{r=1}^{2k} a_{j_{r}}   I_J(i,p) I_J(i,n),
	\end{align*}
	where $J=(j_1,j_2,\ldots,j_{2k})$ and $I(i,J,p,n)$ is as defined in (\ref{eqn:I(i,j,p,n)}).
\end{lemma}
For the proof of Lemma \ref{lem:trac_TT^H}, we refer to the proof of Lemma \ref{lem:trac_TT^t}.
 The following theorem provides the limiting sequence for  $\E\left[\frac{1}{n}\Tr(T_{n\times p} T^*_{n \times p})^r\right]$. % for a Hermitian Toeplitz matrix with independent entries.
\begin{theorem}\label{thm: Moment TT^t Hermitian}
	%Suppose  $T_{n\times p}$ are $n\times p$ Toeplitz matrices with input sequence $\{ a_i\}_{ i \in \mathbb{Z} }$ obeying $(\ref{eqn:condition})$ and such that $a_{-j}=a_{j}$ for all $j \geq 1$ and $a_0 \equiv 0$. Let $X=T_{n\times p} (T_{n \times p})^t$. 
	Let $\{a_i\}_{i \in \N}$ be a sequence of complex random variables which satisfy Assumption II with $a_0 \equiv 0$. Suppose $X=T_{n\times p} T_{n \times p}^*$ and $T_{n \times p}$ is the $n \times p$ Hermitian Toeplitz matrix with input entries  $\{\frac{a_i}{\sqrt{n}}\}_{i \geq 0}$ and $a_{-i}=\bar{a}_i$ for all $i \in \N$. Then for all $r \geq 1$, $\E\left[\frac{1}{n} \Tr X^{r}\right]=M_{r}+o(1/\sqrt{n})$  as $p/n \rightarrow \lambda>0$, $n \rightarrow \infty$, where $M_{r}$ is as defined in (\ref{eqn:2r-Moment_symTnp}).
\end{theorem}
\begin{proof}
	The proof of Theorem \ref{thm: Moment TT^t Hermitian} is similar to the proof of Theorem \ref{thm: Moment TT^t}.  Note that similar calculations and Lemma \ref{lem:trac_TT^H} imply that 
	\begin{equation}\label{eqn:E1/n tr X Hermitian}
		\E\left[\frac{1}{n} \Tr\left(X^{ r}\right)\right]=\frac{1}{n^{r+1}} \sum_{i=1}^n \sum_{\pi \in \widehat{\mathcal{P}}(2r)}\sum_{J \in \Pi(\pi)}  \E \left[\widehat{a_J}\right] I(i,J,p,n) \delta_{0} (\sum_{q=1}^{2r}(-1)^{q} j_{q}),
	\end{equation} 
	where $\widehat{\mathcal{P}}(2r)$ is the set of partitions of $[2r]$ such that each block has size greater than or equal to two, $\Pi(\pi)$ is the set of $J=(j_1,j_2,\ldots , j_{2r}) \in \{-(p-1),\ldots,-1,1,\ldots,(n-1)\}^{2r}$ such that for all $r \sim_{\pi} s$ if and only if $|j_r|=|j_s|$, and $\widehat{a_J}=\prod_{u=1}^{r} \bar{a}_{j_{2u-1}} a_{j_{2u}}$. It follows that a non-zero contribution in the limit occurs only for pair-partitions $\pi$. Furthermore, for a pair-partition $\pi \in \PP_{2}(2r)$, a non-zero contribution in the limit occurs only for 
	$$\Pi_1(\pi)=\{ J \in \Pi(\pi): \text{ for all } u\sim_{\pi}v, \ \epsilon_{\pi}(u)j_u=\epsilon_{\pi}(v)j_v\}.$$
	Let $\{u,v\}$ be a same-parity block of $\pi$ and let $J \in \Pi_1(\pi)$, then it follows from the definition of $\Pi_1(\pi)$ that  $j_u=-j_v$ and subsequently, $\E [a_{j_u}a_{j_v}]=1$ and $\E [\bar{a}_{j_u}\bar{a}_{j_v}]$=1. Similarly, for a different-parity block $\{r,s\}$ of $\pi$ and $J \in \Pi_1(\pi)$, we get that $\E [a_{j_u}\bar{a}_{j_v}]=1$ and $\E [a_{j_u}\bar{a}_{j_v}]$=1. Thus, we get that for every $\pi \in \PP_{2}(\pi)$ and $J \in \Pi_1(\pi)$, $\E[\widehat{a_J}]=1$ and $\sum_{q=1}^{2r}(-1)^qj_q=0$. Therefore (\ref{eqn:E1/n tr X Hermitian}) is equal to (\ref{eqn:E1/n tr X eq3}) and as $n \rightarrow \infty $, $p/n \rightarrow \lambda$, $\E\left[\frac{1}{n} \Tr\left(X^{ r}\right)\right]$ converges to $M_r$ for all $r \geq 1$. 
\end{proof}

The following theorem gives the fluctuation of linear eigenvalue statistics of $T_{n\times p} T_{n \times p}^*$ when $T_{n\times p}$ is  a Hermitian Toeplitz matrix.

\begin{theorem}\label{thm:LES_TT'her}
	Let $\lambda \in (0,\infty)$. Suppose $T_{n \times p}$ is an $n \times p$ complex Hermitian Toeplitz matrix with input entries  $\{\frac{a_i}{\sqrt{n}}\}_{i \geq 0}$, where $\{a_i\}_{i \in \N}$ satisfies Assumption II and $a_0\equiv 0$.
	Then as $ n,p \tends \infty$ with $p/n \tends \lambda$,
	%	\begin{equation}
		%		w_k \stackrel{d}{\rightarrow} N(0,\sigma_k^2),
		%	\end{equation}
	%where $w_k$ is as in (\ref{eqn:wp_Hn_odd}) for non-symmetric Toeplitz matrix  $T_{n\times p}$ and 	$\sigma_p^{2}$ is as in (\ref{eqn:cov_TT'non-sym}).
	%	Moreover, for a given polynomial $Q(x)=\sum_{j=2}^{k}c_{j} x^{j}$ with degree $k\geq 2$, 
	\begin{equation*}
		\frac{1}{\sqrt{n}} \big[ \Tr Q(T_{n \times p}T^*_{n \times p}) - \E[\Tr Q(T_{n \times p}T^*_{n \times p})] \big] \stackrel{d}{\rightarrow} N(0,\sigma_Q^2),
	\end{equation*}
     where $Q(x)=\sum_{j=0}^{k}c_{j} x^{j}$ is a real polynomial with degree $k\geq 1$ and $\sigma_Q^2=\sum_{j_1,j_2=0}^{k}c_{j_1} c_{j_2} \sigma_{j_1,j_2}$ with 
	\begin{align} \label{eqn:sigma_p,q_TT'_her}
		\sigma_{j_1,j_2}= \sum_{\pi \in \mathcal{P}_2(2k_1,2k_2)}  f^{-}_I(\pi) + (\kappa-1) \sum_{\pi \in \mathcal{P}_{2,4}(2k_1,2k_2)}  (f^-_{II}(\pi)+f^+_{II}(\pi)),
	\end{align}
	where $f^-_I(\pi), f^-_{II}(\pi)$ and $f^+_{II}(\pi)$ are as given in (\ref{eqn:f_Ipi-_T}),  (\ref{eqn:f_II-pi_T}) and (\ref{eqn:f_II+pi_T}), respectively.
\end{theorem}

\begin{proof}
	The proof is similar to the proof of Theorem \ref{thm:covarTT'} and here we only mention the differences in the proofs.  The key difference here is that $\E a_j^2=0$ for all $j$. Define $$w^*_{k}:=\frac{1}{\sqrt{n}} \big[ \Tr(T_{n \times p} T^*_{n \times p})^{k} - \E[\Tr(T_{n \times p} T^*_{n \times p})^{k}] \big].$$
	 By Lemma \ref{lem:trac_TT^H}, the covariance structure can be given by 
	 \begin{align}\label{eqn:cov XX* 1}
	 	\Cov(w^*_{k_1},w^*_{k_2}) &=\frac{1}{n^{k_1+k_2+1}}  \sum_{\pi \in \widehat{\mathcal{P}}\left(2k_1+2k_2\right)} \sum_{i=1 \atop i^\prime=1}^{n} |\Pi(\pi)| V(\pi),
	 	 \end{align} 
	 	where 
	 	$$V(\pi) =\Big[\E\Big(\prod_{r=1}^{k_1+k_2} \bar{a}_{j_{2r-1}}a_{j_{2r}}\Big)-\E\Big(\prod_{r=1}^{k_1} \bar{a}_{j_{2r-1}}a_{j_{2r}}\Big)\E\Big(\prod_{r=k_2+1}^{k_2} \bar{a}_{j_{2r-1}}a_{j_{2r}}\Big) \Big] g(i,i^\prime,J_1,J_2,p,n),$$
	with $g(i,i^\prime,J_1,J_2,p,n)$ is as given in (\ref{eqn:g(J,Jprime,p,n)}) and $\widehat{\mathcal{P}}(2k_1+2k_2)$ is the set of all partitions of $[2k_1+2k_2]$ with all blocks having size at least 2. Now, we proceed to find the partitions $\pi$ that make non-zero contribution in the limit for (\ref{eqn:cov XX* 1}). 
 As in Theorem \ref{thm:covarTT'}, it follows that the contribution due to partitions $\pi$ with $|\pi|<k_1+k_2-1$ would be $o(1)$. 
 
 For $|\pi|=k_1+k_2-1$, it follows that a non-zero contribution in the limit arises due to $\pi \in \PP_{2,4}(2k_1+2k_2)$. Furthermore, for $\pi \in \PP_{2,4}(2k_1+2k_2)$ and $J \in \Pi(\pi)$ we get that a non-zero contribution occurs only when $J$ belongs to $\Pi_1(\pi)$ or $ \Pi_2(\pi)$, where $\Pi_1(\pi)$ and $\Pi_2(\pi)$ are as defined in (\ref{eqn: Cov eq 3}). Consider a vector $J$ belonging to $\Pi_1(\pi)$ or $\Pi_2(\pi)$. For a same-parity block $\{r,s\}$ of $\pi$, by the definition of $\epsilon_{\pi}$, it follows that $j_r=-j_s$ and so $\bar{a}_{j_r}=a_{j_s}$. Similarly, for a different-parity block $\{r,s\}$ of $\pi$, we get $a_{j_r}=a_{j_s}$ and $\bar{a}_{j_r}=\bar{a}_{j_s}$. Proceeding in the same fashion, we get that for a block $\{r_1,s_1,r_2,s_2\}$ of the $\pi$, their appropriate product in $\E\Big(\prod_{r=1}^{2k_1+2k_2} \bar{a}_{j_{2r-1}}a_{j_{2r}}\Big)$ is $\kappa$. Thus for $n \rightarrow \infty$ and $p/n \rightarrow \lambda$, the contribution of $\pi$ is $(\kappa-1) (f^-_{II}(\pi)+f^+_{II}(\pi))$.
 
  For pair-partitions $\pi$, since $\E a_j^2=0$, a non-zero contribution occurs only when for all $r \sim_{\pi} s$, $a_{j_r}=\bar{a}_{j_s}$, that is, $j_r= -j_s$. Thus, for $\pi \in \PP_2(2k_1,2k_2)$, a non-zero contribution occurs only when $(J_1,J_2) \in \Pi(\pi)$ obeys (\ref{eqn: epsilon_cond (J_1,J_2)}). Thus the contribution in this case is $f_I^{-}(\pi)$ and this completes the proof. 
\end{proof}

Now we consider Hermitian Toeplitz matrices with complex Brownian motion entries. Firstly, we state the definition of a complex Brownian motion. A complex-valued stochastic process $\{Z(t); t \geq 0\}$ is a \textit{complex Brownian motion} if it can be written as
$$Z(t)= X(t) +\mathrm{i} Y(t),$$
where $\{X(t); t \geq 0\}$ and $\{Y(t); t \geq 0\}$ are independent real Brownian motions. Let  $\{Z_m(t); t \geq 0\}_{m \geq 0}$ be an independent sequence of complex Brownian motion. The following theorem provides the fluctuation of linear eigenvalue statistics of $T_{n\times p}(t) T_{n \times p}^*(t)$ for Hermitian Toeplitz matrix with complex Brownian motion entries.
\begin{theorem}\label{thm:LES_Tp(t)_her} 
	Let $\lambda \in (0,\infty)$ and $\{Z_m(t); t \geq 0\}_{m \geq 0}$ be an independent sequence of complex Brownian motion with $Z_0(t)=0$. Suppose the entries of Hermitian Toeplitz matrix  $T_{n\times p}(t)$ is  $\{\frac{Z_m(t)}{\sqrt{n}}; t \geq 0\}_{m \geq 0}$. Then for $ k\geq 1$, as $ n,p \tends \infty$ with $p/n \tends \lambda$,
	\begin{equation*}
		 \Big\{\frac{1}{\sqrt{n}} \big[ \Tr (T_{n \times p}(t)T^*_{n \times p}(t)) - \E[\Tr (T_{n \times p}(t)T^*_{n \times p}(t))] \big] ; t \geq 0\Big \} \stackrel{\mathcal D}{\rightarrow} \{N_{k}(t) ; t \geq 0\},
	\end{equation*}
	where $\{N_{k}(t);t\geq 0, k\geq 1\}$ are zero mean Gaussian processes with covariance structure as: for $t_1 \leq t_2$,
	\begin{align*} %\label{eqn:covTT'(t)_her}
		&\Cov(N_{k_1}(t_1),N_{k_2}(t_2)) \nonumber \\
		& \ = \sum_{r=1}^{k_2} \binom{2k_2}{2r} t_1^{k_1+r} (t_2-t_1)^{k_2-r} \Big[ \sum_{(\pi_1,\pi_2) \in \mathcal{P}_2(2k_1,2r) \atop \times \mathcal{P}_2(2k_2-2r)}  f_I^-(\pi) + 2 \sum_{(\pi_1,\pi_2) \in \mathcal{P}_{2,4}(2k_1,2r)\atop \times \mathcal{P}_2(2k_2-2r)}  (f^-_{II}(\pi)+f^+_{II}(\pi)) \Big].
	\end{align*}
	%	where $M'_{2(k_2-r/2)}$ is as in (\ref{eqn:2r-moment_non-symT}) and
	%\begin{align*}
	%\theta'_1(k_1,r)& = \sum_{\pi \in \mathcal{SP}_2(2k_1,r)}  g_I(\pi) + 2 \sum_{\pi \in \mathcal{SP}_{2,4}(2k_1,r)}  g_{II}(\pi),
	%\end{align*}
	Here $f^-_I(\pi), f^-_{II}(\pi)$ and $f^+_{II}(\pi)$ are as given in (\ref{eqn:f_Ipi-_T}),  (\ref{eqn:f_II-pi_T}) and (\ref{eqn:f_II+pi_T}), respectively 	with $(\pi_1,\pi_2) \in \mathcal{SP}_2(2k_1,2r) \times \mathcal{SP}_2(2(k_2-r))$ as in Definition \ref{remark:pi1,pi2}.
\end{theorem}

\begin{remark}
	In Theorem \ref{thm:LES_TT'her} and Theorem \ref{thm:LES_Tp(t)_her}, we considered the Hermitian Toeplitz matrices with  zero diagonal entries. For the Hermitian Toeplitz matrices with non-zero diagonal entries, one can conclude similar results as  Theorem \ref{thm:LES_TT'_a0} and Theorem \ref{thm:LES_Tp_b0}.
\end{remark}

\section{\large \bf \textbf{Schatten r-norm of random Toeplitz matrices}} \label{sec:schatten-norm}
%\red{Write in details and clearly way, like, What is $X$; in (i), (ii),, r?? try to define for matrix of order  $n \times p$. What do you mean by results for fluctuation? }
%\red{If you want to add something some, add in this section}
For an $n \times p$ matrix $A$ and $1 \leq r < \infty$, the $r$-Schatten norm is defined as
\begin{equation*}
	||A||_r:=\left(\Tr(AA^*)^{r/2}\right)^{1/r}=\left(\sum_{i} \sigma_i^r\right)^{1/r},
\end{equation*}
where $\sigma_i$ are the singular values of the matrix $A$. It follows that for $w_{k}$ as defined in (\ref{eqn:wp_Hn_odd}), 
\begin{equation*}
	w_k=\frac{1}{\sqrt{n}} \left(||T_{n \times p}||_{2k}^{2k}-\E ||T_{n \times p}||_{2k}^{2k}\right).
\end{equation*}

Our study on linear eigenvalue statistics of $T_{n \times p}T_{n \times p}^\prime$ gives the following results on the $r$-Schatten norms of Toeplitz matrices.
\begin{theorem}
	Let $\lambda \in (0,\infty)$ and  $\{a_i\}_{i \in \Z}$ be a sequence of random variables with $a_0 \equiv 0$. Let $r \in \N$ and suppose $p,n \rightarrow \infty$ and $p/n \rightarrow \lambda$.
	\begin{enumerate}[(i)]
	\item For symmetric Toeplitz matrix $T_{n \times p}$ with input sequence $\{\frac{a_i}{\sqrt{n}}\}_{i \geq 0}$ such that $\{a_i\}_{i \in \N}$ obeys Assumption I, $\frac{1}{n^{1/2r}}||T_{n \times p}||_{2r} \rightarrow M_{r}^{1/2r}$ in probability, where $M_r$ is as defined in (\ref{eqn:2r-Moment_symTnp}).
	
	\item  For non-symmetric Toeplitz matrix $T_{n \times p}$, with input sequence $\{\frac{a_i}{\sqrt{n}}\}_{i \in \Z}$  such that $\{a_i\}_{|i|\geq 1}$ obeys Assumption I, $	\frac{1}{n^{1/2r}}||T_{n \times p}||_{2r} \rightarrow (M_{r}^{\prime})^{1/2r}$ in probability, where $M_r^{\prime}$ is a defined in (\ref{eqn:2r-moment_non-symT}).
	
	\item For Hermitian Toeplitz matrix $T_{n \times p}$, with input sequence $\{\frac{a_i}{\sqrt{n}}\}_{n \geq 0}$ such that $\{a_i\}_{i \in \N}$ obeys Assumption II, $\frac{1}{n^{1/2r}}||T_{n \times p}||_{2r} \rightarrow \left({M_{r}}\right)^{1/2r} $in probability, where $M_r$ is as defined in (\ref{eqn:2r-Moment_symTnp}).
	\end{enumerate}
\end{theorem}
\begin{proof}
	We present only the proof of part $(i)$, the rest of the parts follow from similar arguments. Note that by Theorem \ref{thm: Moment TT^t}, it follows that $\E\left[\frac{1}{n}||T_{n \times p}||_{2r}^{2r}\right]$ converges to $M_{r}$ as $n$ goes to infinity. From Theorem \ref{thm:covarTT'}, it follows that $\frac{1}{n^2}\var\left[||T_{n \times p}||_{2r}^{2r}-\E ||T_{n \times p}||_{2r}^{2r}\right]$ converges to zero as $n \rightarrow \infty$. Thus we get that $\frac{1}{n}||T_{n \times p}||_{2r}^{2r}$ converges in probability to $M_{r}$ for all $r \geq 1$, implying the result.
\end{proof}
The following theorem says about the fluctuation behaviour of Schatten norm of $T_{n \times p}$.
\begin{theorem}
		Let $\lambda \in (0,\infty)$ and  $\{a_i\}_{i \in \Z}$ be a sequence of random variables with $a_0 \equiv 0$. Let $r \in \N$ and suppose $n,p \rightarrow \infty$ and $p/n \rightarrow \lambda$.
	\begin{enumerate}[(i)]
		\item For symmetric Toeplitz matrix $T_{n \times p}$ with input sequence $\{\frac{a_i}{\sqrt{n}}\}_{n \geq 0}$ such that $\{a_i\}_{i \in \N}$ obeys Assumption I,
		 $$\sqrt{n}\left(\frac{||T_{n \times p}||_{2r}}{n^{1/2r}} - M_{r}^{1/2r} \right) \stackrel{d}{\rightarrow}  N\left(0,\frac{1}{4r^2}M_{r}^{\frac{1}{r}-2}\sigma_{r,r}\right),$$
		 where $M_{r}$ is as defined in (\ref{eqn:2r-Moment_symTnp}) and $\sigma_{r,r}$ as in (\ref{eqn:var Toeplitz symm}).  
		
		\item  For non-symmetric Toeplitz matrix $T_{n \times p}$, with input sequence $\{\frac{a_i}{\sqrt{n}}\}_{i \in \Z}$  such that $\{a_i\}_{|i|\geq 1}$ obeys Assumption I, 
		$$\sqrt{n}\left(\frac{||T_{n \times p}||_{2r}}{n^{1/2r}} - {M^\prime_r}^{1/2r}\right) \stackrel{d}{\rightarrow}  N\left(0,\frac{1}{4r^2}{M^\prime_{r}}^{\frac{1}{2r}-1}\widehat{\sigma}_{r,r}\right),$$ 
		where $M'_{r}$ is as defined in (\ref{eqn:2r-moment_non-symT}) and $\widehat{\sigma}_{r,r}$ as in (\ref{eqn:cov_TT'non-sym}).
	
		\item For Hermitian Toeplitz matrix $T_{n \times p}$, with input sequence $\{\frac{a_i}{\sqrt{n}}\}_{n \geq 0}$ such that $\{a_i\}_{i \in \N}$ obeys Assumption II,
		 $$\sqrt{n}\left(\frac{||T_{n \times p}||_{2r}}{n^{1/2r}} - M_{r}^{1/2r} \right) \stackrel{d}{\rightarrow}  N\left(0,\frac{1}{4r^2}M_{r}^{\frac{1}{r}-2}\sigma_{r,r}\right),$$
		 where $M_{r}$ is as defined in (\ref{eqn:2r-Moment_symTnp}) and $\sigma_{r,r}$ as in (\ref{eqn:sigma_p,q_TT'_her}).
	\end{enumerate}
\end{theorem}
\begin{proof}
	We only prove part $(i)$. The rest of the  cases follow from similar argument. Note that Theorem \ref{thm: Moment TT^t} imply $\sqrt{n}\left(\frac{1}{n} \E ||T_{n \times p}||_{2r}^{2r}- M_r \right)$ converges to 0. Therefore, from Theorem \ref{thm:LES_TT'}, it follows that 
	$$\sqrt{n}\left(\frac{||T_{n \times p}||_{2r}^{2r}}{n} - M_{r}\right)$$
	converges in distribution to the normal distribution with mean zero and variance $\sigma_{r,r}$. From Remark \ref{rem: Mr >0}, it follows that $M_r$ is strictly greater than zero for all $r$. Applying Delta method with the function $g(x)=x^{1/2r}$ gives the required result. 
\end{proof}

\noindent\textbf{Acknowledgment:}
The research work of S.N. Maurya is supported by the fund:
NBHM Post-doctoral Fellowship (order no. 0204/10/(25)/2023/R$\&$D-II/2803). This work was partially done during his stay at IISER Bhopal (Funded by DST/INSPIRE/04/2020/000579). \\
%We would like to thank the anonymous reviewer for his comments and valuable suggestions.

\noindent\textbf{Author declaration:} The authors have no conflicts to disclose.\\

\noindent\textbf{Data availability:} Data sharing is not applicable to this article as no new data were created or analysed in this study.

%\bibliography{shambhubib}
%\bibliographystyle{amsplain}

%\providecommand{\bysame}{\leavevmode\hbox to3em{\hrulefill}\thinspace}
%\providecommand{\MR}{\relax\ifhmode\unskip\space\fi MR }
%% \MRhref is called by the amsart/book/proc definition of \MR.
%\providecommand{\MRhref}[2]{%
%	\href{http://www.ams.org/mathscinet-getitem?mr=#1}{#2}
%}
%\providecommand{\href}[2]{#2}

\providecommand{\bysame}{\leavevmode\hbox to3em{\hrulefill}\thinspace}
\providecommand{\MR}{\relax\ifhmode\unskip\space\fi MR }
\providecommand{\MRhref}[2]{%
	\href{http://www.ams.org/mathscinet-getitem?mr=#1}{#2}
}
\providecommand{\href}[2]{#2}

\end{document}